\documentclass[a4paper,11pt]{article}
\usepackage[latin1]{inputenc}
\usepackage{amsfonts}
\usepackage{amsthm}
\usepackage{amsmath}
\usepackage{amsfonts}
\usepackage{latexsym}
\usepackage{amssymb}
\usepackage{dsfont}
\usepackage{graphicx}
\usepackage{float}
\usepackage[usenames]{color}

\newtheorem{fed}{Definition}[section]
\newtheorem*{fed*}{Definition}
\newtheorem*{feds*}{Definitions}
\newtheorem{teo}[fed]{Theorem}
\newtheorem*{teo*}{Theorem}
\newtheorem{lem}[fed]{Lemma}

\theoremstyle{definition}
\newtheorem{rem}[fed]{Remark}

\newtheorem*{rems*}{Remarks}

\newtheorem{exa}[fed]{Example}

\newtheorem{algo}[fed]{Algorithm}
\newtheorem{nota}[fed]{Notation}

\def\coma{\, , \, }

\def\py{\peso{and}}
\newcommand{\peso}[1]{ \quad \text{ #1 } \quad }

\def\n0{n_{ \text{\rm \tiny o}}}

\newcommand{\IN}[1]{\mathbb {I} _{#1}}

\def\suml{\sum\limits}
\def\prodl{\prod\limits}
\def\QEDP{\tag*{\QED}}

\def\bce{\begin{center}}
\def\ece{\end{center}}

\DeclareMathOperator{\FP}{FP\,}

\def\cD{\mathcal D}

\def\dd{{\mathbf d}}

\def\py{\peso{and}}

\def\noi{\noindent}
\def\cF{\mathcal F}
\def\cG{\mathcal G}
\def\QED{\hfill $\square$}
\def\EOE{\hfill $\triangle$}
\def\EOEP{\tag*{\EOE}}
\def\uno{\mathds{1}}
\def\bm{\left[\begin{array}}
\def\em{\end{array}\right]}
\def\ben{\begin{enumerate}}
\def\een{\end{enumerate}}
\def\bit{\begin{itemize}}
\def\eit{\end{itemize}}
\def\barr{\begin{array}}
\def\earr{\end{array}}
\def\igdef{\ \stackrel{\mbox{\tiny{def}}}{=}\ }

\def\la{\lambda}
\def\al{\alpha}
\def\N{\mathbb{N}}
\def\R{\mathbb{R}}
\def\C{\mathbb{C}}
\def\I{\mathbb{I}}

\def\cC{\mathcal{C}}

\def\cH{\mathcal{H}}

\def\cS{{\cal S}}

\def\cM{{\cal M}}
\def\cB{{\cal B}}

\def\cV{{\cal V}}

\def\inc{\subseteq}

\def\ua{^\uparrow}
\def\da{^\downarrow}

 \DeclareMathOperator{\tr}{tr}

\DeclareMathOperator{\convf}{Conv (\R_{\ge0})}
\DeclareMathOperator{\convfs}{Conv_s (\R_{\ge0})}

\DeclareMathOperator*{\pot}{P_\varphi}
\DeclareMathOperator*{\potd}{P_{\varphi\coma \dd}}

\newcommand{\hil}{\mathcal{H}}

\def\varp{\varphi}

\newcommand{\mat}{\mathcal{M}_d(\mathbb{C})}

\newcommand{\matsad}{\mathcal{H}(d)}

\newcommand{\matud}{\mathcal{U}(d)}

\newcommand{\matpos}{\mat^+}

\newcommand{\matrec}[1]{\mathcal{M}_{#1} (\mathbb{C})}

\def\beq{\begin{equation}}
\def\eeq{\end{equation}}

\def\pausa{\medskip\noi}

\def\Ax2{\,( S_{E(\cF)^\#_\cV})\hat{}_x }

\begin{document}

\title{Optimal frame designs for multitasking devices \\ with weight restrictions}
\author{Mar\'\i a J. Benac, Pedro Massey, Mariano Ruiz and Demetrio Stojanoff
\footnote{Partially supported by CONICET
(PICT ANPCyT 1505/15) and  Universidad Nacional de La Plata (UNLP 11X829) 
 e-mail addresses: mjbenac@gmail.com , massey@mate.unlp.edu.ar , mruiz@mate.unlp.edu.ar , demetrio@mate.unlp.edu.ar}
\\
{\small Depto. de Matem\'atica, FCE-UNLP,  La Plata
and IAM-CONICET, Argentina }}
\date{}
\maketitle

\begin{abstract}
\noindent Let $\mathbf d=(d_j)_{j\in\mathbb I_m}\in\mathbb N^m$ be a finite 
sequence (of dimensions) and $\alpha=(\alpha_i)_{i\in\mathbb I_n}$ 
 be a sequence of positive numbers (of weights), where $\mathbb I_k=\{1,\ldots,k\}$ for $k\in\mathbb N$. We introduce
the $(\alpha\, , \,\mathbf d)$-designs i.e., $m$-tuples $\Phi=(\mathcal F_j)_{j\in\mathbb I_m}$ 
such that $\mathcal F_j=\{f_{ij}\}_{i\in\mathbb I_n}$ is a finite sequence in $\mathbb C^{d_j}$, $j\in\mathbb I_m$, 
and such that the sequence of non-negative numbers $(\|f_{ij}\|^2)_{j\in\mathbb I_m}$ forms a 
partition of $\alpha_i$, $i\in\mathbb I_n$. We characterize the existence of $(\alpha\, , \, \mathbf d)$-designs with 
prescribed properties in terms of majorization relations. 
We show, by means of a finite-step algorithm, that there exist 
$(\alpha\, , \, \mathbf d)$-designs $\Phi^{\rm op}=(\mathcal F_j^{\rm op})_{j\in\mathbb I_m}$ 
that are universally optimal; that is, for every convex function $\varphi:[0,\infty)\rightarrow [0,\infty)$ then $\Phi^{\rm op}$ minimizes the joint convex potential induced by $\varphi$ among $(\alpha\, , \, \mathbf d)$-designs, namely 
$$ 
\sum_{j\in\mathbb I_m}\text{P}_\varphi(\mathcal F_j^{\rm op})\leq \sum_{j\in \mathbb I_m}\text{P}_\varphi(\mathcal F_j) 
$$ 
for every $(\alpha\, , \,  \mathbf d)$-design $\Phi=(\mathcal F_j)_{j\in\mathbb I_m}$, where $\text{P}_\varphi(\cF)=\tr(\varphi(S_\cF))$; in particular, $\Phi^{\rm op}$ minimizes both the joint frame potential and the joint mean square error among 
$(\alpha\, , \, \mathbf d)$-designs. 
We show that in this case $\cF_j^{\rm op}$ is a frame for $\C^{d_j}$, for $j\in\I_m$. This corresponds to the existence of optimal encoding-decoding schemes for multitasking devices with energy restrictions.
\end{abstract}


\tableofcontents

\section{Introduction}

A finite sequence $\cF=\{f_i\}_{i\in\I_n}$ of vectors in $\C^d$ is a frame for $\C^d$ if $\cF$ is a (possibly redundant) system of generators  for $\C^d$. 
In this case, it is well known that there exist finite sequences $\cG=\{g_i\}_{i\in\I_n}$ in $\C^d$ - the so called duals of $\cF$ - such that
\begin{equation}\label{eq intro dual}
f=\sum_{i\in\I_n} \langle f\coma f_i\rangle\ g_i = \sum_{i\in\I_n} \langle f\coma g_i\rangle\ f_i \peso{for} f\in\C^d\,.
\end{equation}
Thus, we can encode/decode the vector $f$ in terms of the inner products $(\langle f\coma f_i\rangle)_{i\in\I_n}\in\C^n$: 
(see \cite{TaF,FinFram,Chr} and the references therein). 
%
The frame operator 
$S_\cF\in \matpos$ is given by 
\begin{equation}\label{eq intro defi op frame}
S_\cF f=\sum_{i\in\I_n} \langle f\coma f_i\rangle \ f_i \peso{for} f\in\C^d\,.
\end{equation} 
If $S_\cF$ is invertible (i.e. if $\cF$ is a frame) the canonical dual of $\cF$ is  given by $g_i=S_\cF^{-1} f_i$ for $i\in\I_n$; this dual plays a central role 
in applications since it has several optimal (minimal) properties within the set of duals of $\cF$. Unfortunately, the computation 
of the canonical dual depends on finding $S_\cF^{-1}$, which is a challenging task from the numerical point of view. A way out of this problem
is to consider those frames $\cF$ for which $S_\cF^{-1}$ is easy to compute (e.g. 
tight frames). 
In general, the numerical stability 
of the computation of $S_\cF^{-1}$ depends on the spread of the eigenvalues of $S_\cF$. 
In \cite{BF} Benedetto and Fickus introduced a convex functional called the frame 
potential of a sequence $\cF=\{f_i\}_{i\in\I_n}$ given by 
\begin{equation}
\FP(\cF)=\sum_{i\coma j\in\I_n}|\langle f_i\coma f_j\rangle|^2\geq 0\,.
\end{equation}
In \cite{BF} the authors showed that under some normalization conditions, $\FP(\cF)$ provides an scalar measure of the spread of the eigenvalues of $\cF$. More explicitly, the authors showed that the minimizers of $\FP$ among sequences $\cF=\{f_i\}_{i\in\I_n}$ for which $\|f_i\|=1$, $i\in\I_n$, are exactly the $n/d$-tight frames. It is worth pointing out that these minimizers are 
also optimal for transmission through noisy channels (in which erasures of the frame coefficients may occur, see \cite{BodPau,HolPau}).

\pausa
In some applications of frame theory, we are drawn to consider frames $\cF=\{f_i\}_{i\in\I_n}$   
such  that $\|f_i\|^2=\alpha_i$, $i\in\I_n$, for some prescribed sequence $\alpha=(\alpha_i)_{i\in\I_n}\in (\R_{>0})^n$; this is 
known as the (classical) frame design problem.
In practice, we can think of frames with prescribed norms as designs for 
encoding-decoding schemes to be applied by a device with some sort of energy restrictions (e.g. a device with limited access to energy power):
 in this case, control of the norms of the frame elements 
amounts to control the energy needed to apply the linear scheme.

\pausa
It is then natural to wonder whether there are tight frames with norms prescribed by $\alpha$. This question has motivated
 the study of the frame design problem (see \cite{Illi,Casagregado,CFMP,CMTL,ID,DFKLOW,FWW,KLagregado} 
 and \cite{FMaP,FMP,MR0,MRiS1,MRS13,mrs2,mrs3} for the more general frame completion problem with prescribed norms). 
It is well known that in some cases there are no tight frames in the class of sequences in $\C^d$ with norms prescribed by $\alpha$; in these cases, 
it is natural to consider  minimizers of the frame potential within this class, since the eigenvalues of the frame operator 
 of such minimizers have minimal spread (thus, inducing more stable linear reconstruction processes). These considerations lead to the study of optimal designs with prescribed structure.
 In \cite{CKFT}, the authors compute the structure 
of such minimizers and show it resembles that of tight frames.  

\pausa 
It is worth pointing out that there are other measures of the spread of the spectra of frame operators (e.g. the mean squared error (MSE)). 
It turns out that both the MSE and the FP lie within the class of convex potentials
introduced in \cite{mr2010}. It is shown in \cite{mr2010} that there are  
solutions $\cF^{\rm op}$ to the frame design problem which are {\bf structural} in 
the sense that 
they are minimizers of every convex potential (e.g. MSE and FP) among frames with squared norms prescribed by $\alpha$. A fundamental tool to show the existence 
of such structural optimal frame designs is the so-called majorization in $\R^n$, which is a partial order used in matrix analysis (see \cite{Bhat}). 

\pausa
Motivated originally in the study of 
optimal finitely generated shift invariant systems with norm restrictions, 
for a finitely generated shift invariant subspace of $L^2(\R^d)$ 
(see \cite[Section 4.2.]{BMS1} and also \cite{BMS2}), 
in the present paper we consider extensions of the (classical) frame design 
problems  as follows: given a finite sequence 
(of dimensions) $\dd=(d_j)_{j\in\I_m}\in \N^m$ and a sequence 
(of weights) $\alpha=(\alpha_i)_{i\in\I_n}\in \R_{>0}^n$,
we consider the set $\cD(\alpha\coma \dd)$ of $(\alpha\coma \dd)$-designs. i.e. $m$-tuples
$\Phi=(\cF_j)_{j\in\I_m}$ such that each $\cF_j=\{f_{ij}\}_{i\in\I_n}$ is a finite sequence in $\C^{d_j}$, 
and 
\beq\label{intro eq restrict}
\sum_{j\in\I_m}\|f_{ij}\|^2=\alpha_i \peso{for} i\in\I_n\,.
\eeq
Notice that the restrictions on the norms above involve vectors in the (possibly different) spaces $f_{ij}\in\C^{d_j}$ for $j\in\I_m$. 
The $(\alpha\coma \dd)$-designs appear as the discretizations in the context of   
finitely generated shift invariant systems (see \cite{BMS1}). 
On the other hand, as in the case of frames with prescribed norms, 
$(\alpha\coma \dd)$-designs can be considered as encoding-decoding 
schemes to be applied by a multitasking device with some sort of energy restriction (e.g. due to isolation, or devices that are far from 
energy networks); in case $\cF_j$ is a frame for $\C^{d_j}$ for $j\in\I_m$, then $\Phi=(\cF_j)_{j\in\I_m}$ induces
linear schemes in the spaces $(\C^{d_j})_{j\in\I_m}$ that run in parallel. 
In this case, we want to control the overall energy needed (in each step of the encoding-decoding scheme)
to apply simultaneously the $m$ linear schemes, through the restrictions in Eq.\eqref{intro eq restrict}.

\pausa
It is natural 
to consider those $(\alpha\coma \dd)$-designs that give rise to the more stable multitasking processes. 
In order the measure the overall stability of the family $\Phi=(\cF_j)_{j\in\I_m}$ we can consider the joint 
frame potential of $\Phi$ or the joint MSE of $\Phi$ given by 
$$ \FP(\Phi)=\sum_{j\in\I_m}\FP(\cF_j) \quad , \quad \text{MSE}(\Phi)=\sum_{j\in\I_m} \text{MSE}(\cF_j)\peso{respectively .}
$$
More generally, given a convex function $\varphi:[0,\infty)\rightarrow[0,\infty)$ we introduce the joint convex potential $\pot(\Phi)$ induced by $\varphi$
(see Section \ref{sec modelando} for details); this family of convex potentials (that contains the joint frame potential and joint 
MSE) provides  natural measures of numerical stability of the family $\Phi=(\cF_j)_{j\in\I_m}$. 
We remark that  they are the same potentials  considered in the previously mentioned context of finitely generated shift invariant systems in \cite{BMS1}. 

\pausa
Given $(\alpha\coma \dd)$ as above, in this work we characterize the sequences of positive operators $S_j\in\cM_{d_j}(\C)^+$ for $j\in\I_m$, 
for which there exist $(\alpha\coma \dd)$-designs $\Phi=(\cF_j)_{j\in\I_m}$ 
such that $S_{\cF_j}=S_j$, for $j\in\I_m$.  
Our characterization is obtained in terms of the spectra of the operators $S_j$ and majorization relations, and it extends the well known solution of the classical frame design
problem. 

\pausa
Then, we 
construct $(\alpha\coma \dd)$-designs $\Phi^{\rm op}$ that are optimal in $\cD(\alpha\coma \dd)$; 
in this setting, optimality is measured in terms 
of joint convex potentials, as  discussed above. 
The kernel of this problem is the computation of the optimal spectral structure 
among sequences in $\cD(\alpha\coma \dd)$.

\pausa
We point out that our approach to 
these problems is constructive; indeed, we describe a finite step algorithm that 
produces designs $\Phi\in \cD(\alpha\coma \dd)$ with prescribed spectral structure and optimal designs $\Phi^{\rm op}\in \cD(\alpha\coma \dd)$ as above. 
Moreover, we include several numerical examples of optimal $(\alpha\coma \dd)$-designs obtained
with the implementation of our algorithm in MATLAB.
We further obtain the uniqueness of the spectral structure of  
optimal $(\alpha\coma \dd)$-designs. Moreover, we show that the optimal spectral structure does not depend on the particular choice of the convex potential.
As a consequence, our results generalize the results in \cite{BF,CKFT,mr2010}. 

\pausa
The existence of optimal $(\alpha\coma \dd)$-designs as above settles in the affirmative a conjecture
in \cite[Section 4.2.]{BMS1} regarding the existence of optimal finitely generated shift invariant systems 
(for a finitely generated shift invariant subspace of $L^2(\R^d)$) with norm restrictions, with respect to 
convex potentials (see also \cite{BMS2}). On the other hand, our results have potential applications 
in comunication theory, e.g. in the study of the capacity of Multiple Input Multiple Output (MIMO) 
Additive White Gaussian Noise (AWGN) channels (see \cite{HS07}).

%
\pausa
The paper is organized as follows. In Section \ref{sec2} we recall the notion of majorization together with some
fundamental results about this pre-order. We also include some notions and results related with finite frame theory and convex potentials.
 In Section \ref{sec3} we formalize the notion of $(\alpha\coma \dd)$-designs and describe in detail our main goals. In Section \ref{sec main results} we state and prove
our main results, that include an effective characterization of the existence of 
$(\alpha,\,\dd)$-designs with prescribed spectral structure as well as the 
existence of (universal) optimal designs. 
The paper ends with Section \ref{sec ejems}, in which we present some general comments about the problems we studied, and 
several numerical examples 
that exhibit the properties of the optimal $(\alpha\coma \dd)$-designs computed with a finite step algorithm.

\section{Preliminaries}\label{sec2}
In this section we introduce the notation, terminology and results from matrix analysis and frame theory
that we will use throughout the
paper. General references for these results are the texts \cite{Bhat} and \cite{TaF,FinFram, Chr}.
%
%
In what follows we adopt the following

\pausa
{\bf Notation and terminology}. We let $\mathcal M_{k,d}(\cS)$ be the set of $k\times d$ matrices with coefficients in $\cS\subset \C$ and write $\mathcal M_{d,d}(\C)=\mat$ for the algebra of $d\times d$ complex matrices. We denote by $\matsad\subset \mat$ the real subspace of selfadjoint matrices and by $\matpos\subset \matsad$ the cone of positive semidefinite matrices. We let $\matud\subset \mat$ denote the group of unitary matrices.
For $d\in\N$, let $\I_d=\{1,\ldots,d\}$ and let $\uno_d=(1)_{i\in\I_d}\in\R^d$ be the vector with all its entries equal to $1$.

\pausa
Given $x=(x_i)_{i\in\I_d}\in\R^d$ we denote by $x\da=(x_i\da)_{i\in\I_d}$ (respectively $x\ua=(x_i\ua)_{i\in\I_d}$) 
the vector obtained by rearranging the entries of $x$ in non-increasing (respectively non-decreasing) order. We 
denote by $(\R^d)\da=\{x\da:\ x\in\R^d\}$, $(\R_{\geq 0}^d)\da=\{x\da:\ x\in\R_{\geq 0}^d\}$ and analogously for $(\R^d)\ua$ and $(\R_{\geq 0}^d)\ua$. 

\pausa
Given a matrix $A\in\matsad$ we denote by $\la(A)=\la\da(A)=(\la_i(A))_{i\in\I_d}\in (\R^d)\da$ the eigenvalues of $A$ counting multiplicities and arranged in
non-increasing order, and by $\la\ua(A)$ the same vector  but ordered
in non-decreasing order. 
If $x,\,y\in\C^d$ we denote by $x\otimes y\in\mat$ the rank-one matrix given by $(x\otimes y) \, z= \langle z\coma y\rangle \ x$, for $z\in\C^d$.

\subsection{Majorization}\label{sec2.2}

\pausa Next we recall the notion of majorization between vectors, that will play a central role throughout our work.
\begin{fed}\rm
\ben
\item Let $x,\, y\in\R^d$. We say that $x$ is
{\it submajorized} by $y$, and write $x\prec_w y$,  if
$$
{\suml_{i\in \I_j} x^\downarrow _i\leq \suml_{i\in \I_j} y^\downarrow _i} \peso{for every} 1\leq j\leq d\,.
$$
If $x\prec_w y$ and $\tr x ={\suml_{i\in \I_d} x_i=\suml_{i\in\I_d} y_i }= \tr y$,  then $x$ is
{\it majorized} by $y$, and write $x\prec y$.
\item Let $x\in\R_{\ge 0}^n$ and $y\in\R_{\ge 0}^d$ with $n> d$. 
Then we define the notions of $\prec$ and $\prec_w$ between the vectors $x$ and $y$ 
(of different size) by 
changing $y $ by $y\oplus 0_{n-d}:=(y\coma 0 \coma \dots\coma 0) \in \R^n$. Then 
\beq\label{mayo nd}
x\prec y \peso{if} \suml_{i\in \I_n} x_i=\suml_{i\in\I_d} y_i \py 
\suml_{i\in \I_j} x^\downarrow _i\leq \suml_{i\in \I_j} y^\downarrow _i \peso{for} 1\leq j\leq d 
\ .
\eeq
and similarly one defines $y\prec x$. 
\EOE
\een
\end{fed}

\pausa It is well known that majorization is related with the class $\mathcal{DS}(d)$
 of doubly stochastic matrices i.e., formed by $D\in\mat$ with real non-negative entries such that 
each row sum and column sum equals one. 

\begin{teo}[See \cite{Bhat}]\label{teo prelis majo vs ds} \rm
Let $x,\, y\in\R^d$. Then 
\beq
x\prec y \quad \iff \peso{there exists} D\in\mathcal{DS}(d) \peso{such that} x=Dy \ .
\QEDP
\eeq
\end{teo} 
\begin{rem}\label{rem ds se construye}
Let $x,\, y\in\R^d$ be such that $x\prec y$. Using \cite[Theorem II.1.10]{Bhat} 
we get a finite step algorithm (based on the so-called T-transformations) 
that constructs $D\in \mathcal{DS}(d)$ such that $x=Dy$.

\pausa
Majorization is intimately related with tracial inequalities of convex functions.
The following result summarizes these relations (see for example \cite{Bhat}): \EOE
\end{rem}

\begin{teo}\label{teo intro prelims mayo}
\rm Let $x,\,y\in \R^d$.
If
$\varphi:I\rightarrow \R$ is a
convex function defined on an interval $I\inc \R$ such that
$x,\,y\in I^d$ then:
\ben
\item If $x\prec y$, then
$
\tr \varphi(x) \igdef\suml_{i\in\IN{d}}\varphi(x_i)\leq \suml_{i\in\IN{d}}\varphi(y_i)=\tr \varphi(y)\ .
$
\item If only $x\prec_w y$,  but $\varphi$ is an increasing convex function, then  still
$\tr \varphi(x) \le \tr \varphi(y)$.
\item If $x\prec y$ and $\varphi$ is a strictly convex function such
that $\tr \,\varphi(x) =\tr \, \varphi(y)$ then, $x\da=y\da$.\qed
\een 
\end{teo}

\subsection{Frames and convex potentials}

\pausa
In what follows we adopt the following 

\pausa {\bf Notation and terminology}: let
$\cF=\{f_i\}_{i\in\I_n}$ be a finite sequence in $\C^d$. Then,
\ben \item $T_\cF\in \cM_{d,n}(\C)$ is the synthesis operator
 given by
$T_\cF\, x=\sum_{i\in\I_n}x_i\, f_i$, for $x=(x_i)_{i\in\I_n}\in \C^n$.
\item $T_\cF^*\in \cM_{n,d}(\C)$  is the analysis operator 
and it is given by $T_\cF^*\, f=(\langle
f,f_i\rangle)_{i\in\I_n}$, for $f\in \C^d$.
\item  $S_\cF\in \matpos$ denotes the
frame operator of $\cF$ and it is given by $S_\cF=T_\cF\,T_\cF^*$.
Hence, $${S_{\cF}}\,  f=\sum_{i\in\I_n} \langle f,f_i\rangle
f_i=\sum_{i\in\I_n} (f_i\otimes f_i)\, f \peso{for} f\in\C^d\,.$$ \item We
say that $\cF$ is a frame for $\C^d$ if it spans $\C^d$;
equivalently, $\cF$ is a frame for $\C^d$ if $S_\cF$ is a positive
invertible operator acting on $\C^d$. 
 \een

\pausa
In several applied situations it is desired to construct
a finite sequence $\cG=\{g_i\}_{i\in\I_n}\in (\C^d)^n$, in
such a way that the spectra of the frame
operator of $\cG$ is given by some $\la\in(\R_{\ge 0}^d)\da$ and the squared norms of the frame elements
are prescribed by a sequence of positive numbers  $\alpha=(\alpha_i)_{i\in\I_n}$.
This is known as the (classical)
frame design problem and it has been studied by several research groups (see for 
example \cite{Illi,Casagregado,CFMP,CMTL,ID,DFKLOW,FWW,KLagregado}). The following result
characterizes the existence of such frame designs in terms of majorization relations.

\begin{teo}[\cite{Illi,MR0}]\label{frame mayo}\rm
Let $\lambda\in \R_{\ge0}^d$ and consider
$a=(a_i)_{i\in\I_n}\in(\R_{>0}^n)\da$. Then there exists
a sequence $\cG=\{g_i\}_{i\in\I_n}$ in $\C^d$ 
such that  $\la(S_\cG)= \la\da$  and  $\|g_i\|^2=a_i$ for $i\in\I_n$
if and only if $a\prec\la$. \QED
\end{teo}

\pausa
The previous result shows the flexibility of structured frame designs, which is important in applied situations.
Also, numerical stability of the encoding-decoding scheme induced by a frame plays a role in applications; hence,
a central problem in this area is to described the structured frame designs that maximize the stability of their
encoding-decoding scheme. One of the most important (scalar) measures of stability  is the so-called frame potential
introduced by Benedetto and Fickus in \cite{BF} given by
$$ \FP(\cF)=\sum_{i\coma j\in\I_n}|\langle f_i\coma f_j\rangle|^2=\tr(S_\cF^2) \peso{for} \cF=\{f_i\}_{i\in\I_n}\in(\C^d)^n\,.$$
Benedetto and Fickus have shown that (under certain normalization conditions) minimizers of the frame potential induce the
most stable encoding-decoding schemes.
More generally, we can measure the stability of the scheme induced
by the sequence $\cF=\{f_i\}_{i\in\I_n}\in(\C^d)^n$ in terms of
convex potentials. In order to introduce these potentials we
consider the sets
$$
\convf = \{
\varphi:\R_{\geq 0} \rightarrow \R_{\geq 0}\ :\  \varphi   \ \mbox{ is a convex function} \ \}
$$
and $\convfs = \{\varphi\in \convf : \varphi$ is strictly convex $\}$.
\begin{fed}\label{pot generales}\rm
Following \cite{mr2010} we consider the convex potential
$\pot$ associated to $\varphi\in \convf$, given by
$$
\barr{rl}
\pot(\cF)&=\tr \, \varp(S_\cF) = \sum_{i\in\I_d}\varphi(\lambda_i(S_\cF)\,) \peso {for}
\cF=\{f_i\}_{i\in\I_n}\in  (\C^d)^n \ , \earr
$$
where the matrix $\varp(S_\cF)$ is defined by means of the usual functional calculus. \EOE
\end{fed}
\pausa
Convex potentials allow us to model several well known measures of stability considered in frame theory. For example,
in case $\varp(x)=x^2$ for $x\in\R_{\geq 0}$ then $\pot$ is the Benedetto-Fickus frame potential; in case
$\varp(x)=x^{-1}$ for $x\in \R_{>0}$ then $\pot$ is known as the mean squared error (MSE).

\pausa Going back to the problem of stable designs, it is worth
pointing out the existence of structured designs that are
optimal with respect to every convex potential. Indeed, given
$\alpha=(\alpha_i)_{i\in\I_n}\in \R_{\geq 0}^n$ and $d\in\N$ with $d\leq n$, 
the $\alpha$-torus is the set:
\beq\label{eq defi Balfad}
 \mathcal B_{\alpha\coma d}=\{ \cF=\{f_i\}_{i\in\I_n}\in (\C^d)^n:\ \|f_i\|^2=\alpha_i \, , \ i\in\I_n\}\, .
\eeq We endow $\mathcal B_{\alpha\coma d}$ (which is a product
space) with the product metric. The structure of (local)
minimizers of convex potentials in $\mathcal B_{\alpha\coma d}$
has been extensively studied. The first results were obtained for
the frame potential in \cite{BF} and in a more general context in
\cite{CKFT}. The case of general convex potentials was studied in
\cite{FMaP,FMP,MRiS1,MR0,mr2010,MRS13,mrs2,mrs3} (in some cases in
the more general setting of frame completion problems with
prescribed norms).

\section{On $(\alpha\coma \dd)$-design problems}\label{sec3}

\pausa
We begin this section by introducing notation 
and terminology that allow us to model the $(\alpha\coma \dd)$-design problems, including the optimal design problem with prescribed weights.
 Then, we state and prove our main results for $(\alpha\coma \dd)$-designs.

\pausa
\subsection{Modeling the problem}\label{sec modelando}
Now we generalize the $\alpha$-torus to the multi-frames:
\begin{fed}\label{nueva def}\rm
Let $\alpha=(\alpha_i)_{i\in\I_n}\in \R_{>0}^n$ and $\dd =(d_j)_{j\in\I_m}\in(\N^m)\da$ be such that $d_1\leq n$. 
\ben
\item 
An $(\alpha\coma \dd)$-design is an $m$-tuple
$$
\Phi=(\cF_j)_{j\in\I_m} \ \ , \peso{where}
\cF_j=\{f_{ij}\}_{i\in\I_n} \in (\C^{d_j})^n  \peso{for} j\in\I_m
$$ 
and such that 
$\suml_{j\in\I_m}\| f_{ij}\|^2=\alpha_i$, for $ i\in\I_n\,.
$ 
\item 
We denote by $\mathcal D(\alpha\coma \dd)$ the set of all $(\alpha\coma \dd)$-designs. 
We point out that (in order to simplify our description of the model) we consider 
$(\alpha\coma \dd)$-designs in a broad sense; namely, if
$\Phi=(\cF_j)_{j\in\I_m}\in \mathcal D(\alpha\coma \dd)$ then $\cF_j$ is not necessarily a frame for $\C^{d_j}$, for $j\in\I_m\,$.
\item 
In order to compare the overall stability of the linear encoding-decoding schemes induced by 
an $(\alpha\coma \dd)$-design we introduce the following potentials: 
Given
$\varphi\in\convf$ we consider the joint potential induced by $\varphi$ on $\Phi=(\cF_j)_{j\in\I_m}\in\cD(\alpha\coma \dd)$ given by
\beq
\pot(\Phi)=\sum_{j\in\I_m} \pot(\cF_j)\ =\sum_{j\in\I_m} \tr\ \varphi(S_{\cF_j})
=\sum_{j\in\I_m}\sum_{i\in \I_{d_j}}\varphi(\la_i(S_{\cF_j})\, )\ . \EOEP
\eeq
\een
\end{fed}

\pausa
Consider
the notation and terminology of Definition \ref{nueva def}. 
We can now describe the main problems that we consider in this work as follows: 
\begin{itemize}
\item[P1.] Determine necessary and sufficient conditions for the existence of $(\alpha\coma \dd)$-designs with prescribed spectral structure and describe
algorithmic procedures to construct such designs, in case they exist.
\item[P2.] Given $\varphi\in\convf$ determine the existence and structure of those $\Phi_\varphi\in \mathcal D(\alpha\coma \dd)$
that  minimize the joint convex potential  $\pot$ in $\cD(\alpha\coma \dd)$, that is
\beq\label{eq desi estruc opt}
 \pot(\Phi_\varphi)=\min\{\pot(\Phi):\ \Phi\in \mathcal D(\alpha\coma \dd)\}\, .
\eeq In this case we say that $\Phi_\varphi$ is an $\pot$-optimal $(\alpha\coma \dd)$-design. 
Determine whether these $\pot$-optimal $(\alpha\coma\dd)$-designs depend on the particular choice of $\pot$, for 
strictly convex functions $\varphi\in \convfs$.
\item[P3.] Describe an algorithmic procedure that computes $\pot$-optimal $(\alpha\coma \dd)$-designs. 
\item[P4.] Characterize the $\pot$-optimal $(\alpha\coma \dd)$-designs in terms of some structural properties.
\end{itemize}
We will solve problems P1.-P4. In particular, we will show
that if $\Phi_\varphi=(\cF_j)_{j\in\I_m}$ is an $\pot$-optimal $(\alpha\coma \dd)$-design for $\varphi\in\convfs$ then, $\cF_j$ is a frame for $\C^{d_j}$ 
for each $j\in\I_m$ (see Section \ref{sec main results}). Moreover, we will show that 
$\pot$-optimal $(\alpha\coma\dd)$-designs do not depend on the particular choice of $\pot$, for 
strictly convex functions $\varphi\in \convfs$.
\subsection{Main results}\label{sec main results}

In this section we 
state and prove our main results; these include the existence of 
$(\alpha\coma \dd)$-designs with prescribed spectral structure, and designs with some special structure 
which turn out to be optimal designs in the sense
of Problem (P2). We further show the uniqueness of the spectral structure of 
optimal $(\alpha\coma \dd)$-designs.

\pausa
Our first main result characterizes the existence of $(\alpha\coma \dd)$-designs with prescribed spectral structure. 
We formalize problem P1. in terms of the following 

\begin{fed}\label{defi admiss}
\rm 
Let $\dd=(d_j)_{j\in\I_m}\in (\N^m)\da$ and $\alpha=(\alpha_i)_{i\in\I_n}\in (\R_{\geq 0}^n)\da$ be such that $d_1\leq n$.
Let 
$$
\mu_j=(\mu_{i,j})_{i\in\I_{d_j}}\in (\R_{\geq 0}^{d_j})\da \peso{for} j\in\I_m  
\peso{and set} \cM:=\{\mu_j\}_{j\in\I_m}\in \prod_{j \in \I_m} \, (\R_{\geq 0}^{d_j})\da \ .
$$ 
We say that the pair $(\alpha\coma \cM)$ is {\bf admissible} if there exists 
$\Phi=(\cF_j)_{j\in\I_m}\in\cD(\alpha\coma \dd)$ such that 
$$
\la(S_{\cF_j})=\mu_j  \peso{for every} j\in\I_m  \ .
$$
In this case we denote $\cM = \cM_\Phi\,$. \EOE
\end{fed}

\pausa
In order to obtain  an effective characterization of
admissibility, we introduce the notion of $(\alpha\coma m)$-weight partition matrix.

\begin{rem}\label{defi conjuntos para problemas}\rm
Let $\dd=(d_j)_{j\in\I_m}\in (\N^m)\da$ and $\alpha=(\alpha_i)_{i\in\I_n}\in (\R_{\geq 0}^n)\da$ be such that $d_1\leq n$.
\begin{enumerate}
\item We consider the set of $(\alpha\coma m)$-weight partitions given by
$$
P_{\alpha\coma m}=\{A=(a_{ij})_{i\in\I_n,\, j\in\I_m}\in \mathcal M_{n,m}(\C)\,:\, a_{ij}\geq 0 \quad \text{and}\; 
\sum_{j\in\I_m}a_{ij}=\alpha_i \peso{for} i\in\I_n\,\}\, .$$
\item A sequence $(\cF_j)_{j\in\I_m}\in\cD(\alpha\coma \dd) \iff $  its matrix of weights 
$$
A = \Big\{\|f_{ij}\|^2\Big\}_{i\in\I_n \ j\in \I_m} \ \in P_{\alpha\coma m} \ .
$$
\end{enumerate}
\end{rem}

\begin{lem}[A first characterization of admissible pairs]\label{rem: primera equiv admis y part mayo}
Consider a pair $(\alpha\coma \cM)$ as in Definition \ref{defi admiss} above. Then the following conditions are equivalent. 
\ben
\item The pair $(\alpha\coma \cM)$ is admissible.
\item There exists  a matrix $A\in P_{\alpha \coma m}$ such that  
\beq \label{col}
c_j(A)\prec \mu_j \peso{for every} j\in\I_m \ ,
\eeq
where $c_j(A)\in \R_{\geq 0}^n$ denotes the $j$-th column of $A$.
\een
In particular, the set of sequences $\cM$ such that $(\alpha\coma \cM)$ is admissible is {\bf convex} 
in  $\prodl_{j \in \I_m} \, \R^{d_j}$. 
\end{lem}
\proof
If the pair $(\alpha\coma \cM)$ is admissible, 
let $(\cF_j)_{j\in\I_m}\in\cD(\alpha\coma \dd)$ be such that 
$\la(S_{\cF_j})=\mu_j$, for $j\in\I_m$. Let $A\in P_{\alpha\coma m}$ be given by $c_j(A)=(\|f_{ij}\|^2)_{i\in\I_n}$, where
$\cF_j=\{f_{ij}\}_{i\in\I_n}$, for $j\in\I_m$. Then, by the Theorem \ref{frame mayo}, $c_j(A)\prec \mu_j$ for $j\in\I_m$. 

\pausa
Conversely, assume that there exists 
$A\in P_{\alpha,m}$ with $c_j(A)\prec \mu_j$, for $j\in\I_m$. Then, again by Theorem \ref{frame mayo},
 for each $j\in\I_m$ there exists $\cF_j=\{f_{ij}\}_{i\in\I_n}\in (\C^{d_j})^n$ such that 
$c_j(A)=(\|f_{ij}\|^2)_{i\in\I_n}$ and $\la(S_{\cF_j})=\mu_j$. Hence, $(\cF_j)_{j\in\I_m}\in\cD(\alpha\coma \dd)$ which shows that the pair $(\alpha\coma \cM)$ is admissible.\QED

\pausa
The following result provides an
effective method to determine whether a given pair is admissible or not.

\begin{teo}\label{teo: caract de la part mayo}
Let $\dd=(d_j)_{j\in\I_m}\in (\N^m)\da$ and $\alpha=(\alpha_i)_{i\in\I_n}\in (\R_{\geq 0}^n)\da$ be such that $d_1\leq n$.
Given a sequence  $\cM=\{\mu_j\}_{j\in\I_m}\in \prod_{j \in \I_m} \, (\R_{\geq 0}^{d_j})\da$, 
set
\beq\label{sigmaM}
\sigma_\cM:=\sum_{j\in\I_m} (\mu_j\oplus 0_{d_1-d_j})\in(\R_{\geq 0}^{d_1})\da\ .
\eeq
Then, the pair $(\alpha\coma \cM)$ is admissible if and only if
$\alpha\prec \sigma_\cM$.
\end{teo}
\begin{proof}
Assume first that the pair $(\alpha\coma \cM)$ is admissible. Then, by Lemma \ref{rem: primera equiv admis y part mayo}, 
there exists $A\in P_{\alpha\coma m}$ such that $c_j(A)=(a_{ij})_{i\in\I_n}\prec \mu_j
=(\mu_{i,j})_{i\in\I_{d_j}}\in (\R^{d_j}_{\geq 0})\da$, for $j\in\I_m\,$. 
Fix an index $j\in\I_m\,$. Note that $n\ge d_1 = \min \{n\coma d_1\}$. Hence, by hypothesis, for $k\in\I_{d_1}$ we have that 
\beq\label{eq. cond. nec}
\barr{rl} &
 \suml_{i=1}^k a_{ij}\leq \suml_{i=1}^k (c_j(A)\da)_i\leq \suml_{i=1}^{\min\{k\coma d_j\}} \mu_{i,j} \implies
 \suml_{i=1}^k \alpha_i= \suml_{j\in\I_m} \suml_{i=1}^k a_{ij}\leq \suml_{j\in\I_m} \suml_{i=1}^{\min\{k\coma d_j\}} \mu_{i,j}\ .
\earr
\eeq
Notice that 
$$ \sigma_\cM=(\sigma_i)_{i\in\I_{d_1}}=\sum_{j\in\I_m} (\mu_j\oplus 0_{d_1-d_j})
\implies \sum_{j\in\I_m} \sum_{i=1}^{\min\{k\coma d_j\}} \mu_{i,j}=\sum_{i=1}^k\sigma_i\ \ , \ \ k\in\I_{d_1}\,. $$
Then, Eq. \eqref{eq. cond. nec} shows that $\alpha\prec \sigma_\cM$ (the equality $\tr\, \sigma_\cM 
= \tr\, \alpha$ is clear). 

\pausa
For the converse, in order to show that $(\alpha\coma \cM)$ is admissible, we prove that there exists $A\in P_{\alpha,m}$
such that $c_j(A)\prec \mu_j$, for $j\in\I_m$ (see Lemma \ref{rem: primera equiv admis y part mayo}). Indeed, 
since $\alpha\prec \sigma_\cM$, by Theorem \ref{teo prelis majo vs ds} there exists a doubly stochastic matrix 
$D\in \mathcal{DS}(n)$, such that $D(\sigma_\cM\oplus 0_{n-d_1})=\alpha$. Consider $A\in\cM_{n\coma m}(\R_{\geq 0})$ determined by
$c_j(A)= D(\mu_j\oplus 0_{n-d_j})$, for $j\in\I_m$. Notice that in this case by construction, 
$c_j(A)\prec \mu_j\oplus 0_{n-d_j}\implies c_j(A)\prec \mu_j$ for $j\in \I_m$ and
$$ 
A\,\uno_m=\sum_{j\in\I_m} c_j(A)=\sum_{j\in\I_m}D(\mu_j\oplus 0_{n-d_j})=D(\sigma_\cM\oplus 0_{n-d_1})=\alpha\ .
$$
Thus, $A\in P_{\alpha, m}$ and the pair $(\alpha\coma \cM)$ is admissible.
\end{proof}

\begin{rem}[Finite-step algorithm for constructing $(\alpha\coma \dd)$-designs with prescribed spectral structure]
\label{rem: algotutti} With the notation of 
Theorem \ref{teo: caract de la part mayo}, assume that $\alpha\prec \sigma_\cM\,$.
Hence, in this case the pair $(\alpha\coma \cM)$ is admissible. 
By Remark \ref{rem ds se construye}, there is a finite step algorithm that constructs $D\in\mathcal{DS}(n)$ such that 
$\alpha=D\, (\sigma_\cM\oplus 0_{n-d_1})$.
From the previous proof we see that if we consider $A\in\cM_{n\coma m}(\R_{\geq 0})$ determined by
$c_j(A)= D(\mu_j\oplus 0_{n-d_j})$, for $j\in\I_m$, then $A\in P_{\alpha, m}$. Moreover, by construction we have that 
$c_j(A)\prec \mu_j\oplus 0_{n-d_j}$ ($\implies  c_j(A)\prec \mu_j$) for $j\in \I_m\,$. 

\pausa
We can now apply finite step algorithms (such as the one-sided Bendel-Mickey algorithm, see \cite{CFMP,CMTL,ID,FWW}) 
and obtain $\cF_j=\{f_{ij}\}_{i\in\I_n}\in (\C^{d_j})^n$ such that $(\|f_{ij}\|^2)_{i\in\I_n}=c_j(A)$
and such that $\la(S_{\cF_j})=\mu_j$ for $j\in\I_m$. Therefore, we get 
$\Phi=(\cF_j)_{j\in\I_m}\in\cD(\alpha\coma \dd)$ such that $\cM = \cM_\Phi$ in a constructive way.
\EOE
\end{rem}

\pausa
The following definition introduces an $m$-tuple of vectors (of eigenvalues), associated to every $(\alpha\coma \dd)$-design, and a large vector constructed from the juxtaposition of the elements of this set. These are going to be useful in proving  the existence of optimal designs in terms of majorization relations, related with problem P2 above.  (see Theorem \ref{teo el algo sirve} below).

%

\begin{fed}\label{rem equiv mayo} \rm 
Let $\dd=(d_j)_{j\in\I_m}\in (\N^m)\da$ and $\alpha=(\alpha_i)_{i\in\I_n}\in (\R_{\geq 0}^n)\da$ be such that $d_1\leq n$.
 Let $\Phi=(\cF_j)_{j\in\I_m}\in \mathcal D(\alpha\coma \dd)$ and
let $S_j=S_{\cF_j}\in\cM_{d_j}(\C)^+$ denote the frame operators of $\cF_j$, for $j\in\I_m$. We define 
\beq\label{el Lambda}
\cM _\Phi = \{\lambda(S_j)\}_{j\in\I_m} \in \prod_{j \in \I_m} \, (\R_{\geq 0}^{d_j})\da
\py 
\Lambda_\Phi=\big(\lambda(S_1)\coma\ldots\coma\lambda(S_m)\,\big)\in \R_{\geq 0}^{d}
 \ ,
\eeq
where $d=\tr \, \dd = 
\suml_{j\in\I_m} d_j$ and each $\lambda(S_j)\in(\R_{\geq 0}^{d_j})\da$ is
the vector of eigenvalues of $S_j$, for $j\in\I_m$. Recall that given a generic sequence 
$\cM  \in \prod_{j \in \I_m} \, (\R_{\geq 0}^{d_j})\da$, we say that the pair 
$(\alpha\coma \cM)$ is  admissible  if 
there exists 
$\Phi\in \mathcal D(\alpha\coma \dd)$ such that $\cM = \cM_\Phi\,$, which in turns is equivalent to $\sigma_\cM\prec \alpha$. We also remark that $\Lambda_\Phi$ is not an ordered vector. We shall use the specific 
order of its entries given in Eq.\eqref{el Lambda} in order to preserve the convexity properties
given by Lemma \ref{rem: primera equiv admis y part mayo}. \EOE
\end{fed}

\begin{rem}\label{rem equiv mayo25} Consider the notation in Definition \ref{rem equiv mayo}.  
If $\varphi\in\convf$ and $\pot$ denotes the joint convex potential induced by $\varphi$ (see Definition \ref{pot generales}) then, 
\beq\label{eq igual pot} 
\pot(\Phi) =\sum_{j\in\I_m} \pot(\cF_j)= \sum_{j\in\I_m} \tr(\varphi(\la(S_j)))=\sum_{\ell\in\I_{|d|}}\varphi((\Lambda_\Phi)_\ell)=:\tr(\varphi(\Lambda_\Phi))\,. 
\eeq
Therefore, by Theorem \ref{teo intro prelims mayo} and Eq. \eqref{eq igual pot}, the existence of an (optimal) 
$(\alpha\coma \dd)$-design satisfying Eq. \eqref{eq desi estruc opt}
for every $\varphi\in\convf$ is equivalent
to the existence of $\Psi=(\cG_j)_{j\in\I_m}\in \mathcal D(\alpha\coma \dd)$
such that 
\beq 
\Lambda_{\Psi}\prec \Lambda_\Phi \peso{for every} \Phi=(\cF_j)_{j\in\I_m}\in \mathcal D(\alpha\coma \dd)\, .  \EOEP
\eeq
\end{rem}

\begin{rem}\label{rem sobre la forma y el fondo}
Consider the notation in Definition \ref{rem equiv mayo}. 
In the rest of this section we shall show 
the existence of $(\alpha\coma \dd)$-designs $\Phi^{\rm op}=(\cF_j^{\rm op})_{j\in\I_m}$ that are optimal with respect to 
every joint convex potential (see Theorem \ref{teo el algo sirve con potenciales}). It turns out that these optimal designs have some special features. 

\pausa
In this remark we describe the special structure of the associated sequence $\cM_{\Phi^{\rm op}}$ 
(and introduce the necessary notation to describe this structure)  
in order to make more intelligible the next statements, which are intended to construct admissible pairs with this (optimal) structure. 
Let
$$
\mu_j^{\rm op}=(\mu_{ij}^{\rm op})_{i\in\I_{d_j}}=\la(S_{\cF_j^{\rm op}})\in (\R_{\geq 0}^{d_j})\da 
\peso{for every}  j\in \I_m \ .
$$ 
denote the eigenvalues of the frame operators of $\cF_j^{\rm op}$. 
Then they must have  the following structure: Each vector  
$\mu_j^{\rm op}\in (\R_{\geq 0}^{d_j})\da $ is a (truncated) copy of the first vector 
$\mu_1^{\rm op}\in (\R_{\geq 0}^{d_1})\da $, i.e. 
\beq\label{eq prop espec opti}
 \mu_{ij}^{\rm op} =\mu_{i1}^{\rm op} \peso{for every} 
 i\in \I_j \peso{and every} j\in \I_m \ .
\eeq
In detail,  
let $\sigma(S_{\cF_1^{\rm op}})=\{\gamma_1,\ldots,\gamma_p\}$, with 
$\gamma_1>\ldots>\gamma_p\geq 0$. 
Then there exist 
indexes $g_0=0<g_1<\ldots<g_p=d_1$ (that we shall construct looking for admissibility) such that 
$$ \{ i\in\I_{d_1}: \mu_{i1}=\gamma_\ell\}=\{ i:\ g_{\ell-1}+1 \leq i\leq g_\ell\} \peso{for}\ell\in\I_p\ .$$
We define the following constants, which only depend on the data $\dd=(d_j)_{j\in\I_m}\in (\N^m)\da$: 
\beq\label{eq defi hi}
h_i:=\#\{j\in\I_m:\ d_j\geq i\}\peso{for} i\in \I_{d_1}\,.\eeq
Notice that, 
\beq\label{eq defi h}
h:=(h_i)_{i\in \I_{d_1}}=
\sum_{i=1}^m \ \uno_{d_i}\oplus 0_{d_1-d_i}\,.
\eeq
Using the relations in Eq. \eqref{eq prop espec opti}  we get that
\beq\label{eq defi r}
 \Lambda_{\Phi^{\rm op}}\da=(\gamma_\ell\, \uno_{r_\ell})_{\ell\in\I_p} \peso{where} r_\ell=
\sum_{i=g_{\ell-1}+1}^{g_\ell} h_i \ , \ \ \ell \in\I_p\,.
\eeq
We give an example of this situation  for $m=4$ and $\dd=(6,5,4,2)$ in Figure \ref{grafico del gama op}. 
\begin{figure}[H]
\begin{center}
\scalebox{0.50} {\includegraphics{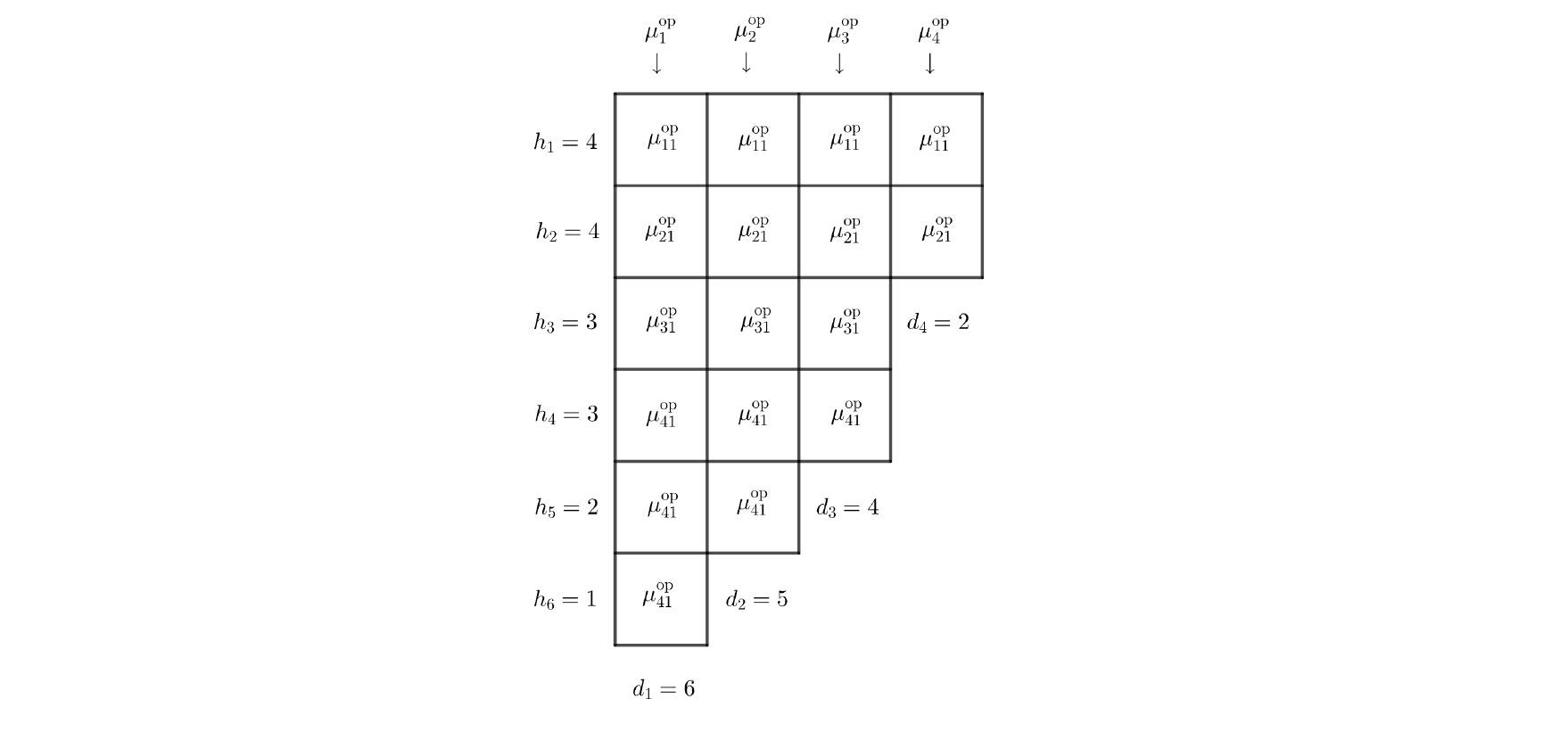}}
\caption{  A graphic example of the structure of $(\mu^{\rm op}_j)_{j\in\I_4}\,$. 
}
\label{grafico del gama op}
\end{center}
\end{figure}

\noi
For example, 
if we assume that 
$$\mu_{11}^{\rm op}=\mu_{21}^{\rm op}=\mu_{31}^{\rm op}=\gamma_1 \ \coma \ \mu_{41}^{\rm op}=\mu_{51}^{\rm op}=\gamma_2 \py  
\mu_{61}^{\rm op}=\gamma_3
\peso{with} \gamma_1>\gamma_2>\gamma_3$$
then we have: $g_0=0$, $g_1=3$, $g_2=5$ and hence, $r_1=11$, $r_2=5$, $r_3=1$; therefore, we compute 
$\Lambda_{\Phi^{\rm op}}\da=(\gamma_1\,\uno_{11}\coma \gamma_2\, \uno_{5}\coma \gamma_3\,\uno_1)\in\R_{>0}^{17}$ in this case.
\EOE
\end{rem}

\pausa
In order to obtain the our next main result, we consider the following

\begin{nota}
Let $\dd=(d_j)_{j\in\I_m}\in (\N^m)\da$ and $\alpha=(\alpha_i)_{i\in\I_n}\in (\R_{\geq 0}^n)\da$ be such that $d_1\leq n$.
Let $h=(h_i)_{i\in \I_{d_1}}$ be defined as in Eq. \eqref{eq defi hi}. 
For $1\leq s\leq t\leq d_1$, denote by $P_{s,t}$ and $Q_{t}$ the ratios
\begin{equation}\label{promedios iniciales}
 P_{s\coma t}=\frac{\suml_{i=s}^t \alpha_i}{\suml_{i=s}^t h_i} \py Q_{t}=\frac{\suml_{i=t}^n \alpha_i}{\suml_{i=t}^{d_1} h_i}\,.
\end{equation}
\end{nota} 

\pausa
The following result is a technical construction that we will use to build
optimal $(\alpha\coma \dd)$-designs with a spectral picture as in Remark \ref{rem sobre la forma y el fondo}.

\begin{teo}\label{defi del pegoteo de espectros}
Let $\dd=(d_j)_{j\in\I_m}\in (\N^m)\da$ and $\alpha=(\alpha_i)_{i\in\I_n}\in (\R_{\geq 0}^n)\da$ be such that $d_1\leq n$.
Consider
the vector $h=(h_i)_{i\in \I_{d_1}}$ as in Eq. \eqref{eq defi h}.
Then, there exist 
$$
p\in\I_{d_1}  \peso{and} g_1,\ldots,g_p\in\N  \peso{with} 
0=g_0< g_1<\cdots<g_p=d_1
$$ 
such that, if we define
$\gamma_i=P_{g_{i-1}+1\coma g_i}$, for $i\in \I_{p-1}$ and $\gamma_p=Q_{g_{p-1}+1}$ according with Eq. \eqref{promedios iniciales}, 
then 
\begin{enumerate}
\item \label{orden de los gamas}$\gamma_1>\ldots>\gamma_p> 0$;
\item
They satisfy the following ``block" majorizations:
\beq\label{mayo en los bloques} 
\barr{rl}
(\gamma_i\, h_k)_{k=g_{i-1}+1}^{g_i} & \succ (\alpha_i)_{k=g_{i-1}+1}^{g_i} 
\peso{for} i\in \I_{p-1} \py \\ &\\
(\gamma_p\, h_k)_{k=g_{p-1}+1}^{d_1} &\succ (\alpha_i)_{k=g_{p-1}+1}^n  \ . \earr
\eeq
In particular, 
\beq\label{admite}
(\gamma_1\, \uno_{g_1-g_0}\coma \gamma_2\, \uno_{g_2-g_1}\coma \ldots\coma 
 \gamma_p\, \uno_{g_p-g_{p-1}})\circ h\succ \alpha  \ , 
\eeq
where $\circ$ denotes the entry-wise product.
\end{enumerate}
\end{teo}
\begin{proof}
First note that  $d_1$ clearly satisfies $Q_{d_1}\geq P_{d_1\coma  d_1}$.
Then we can define the index 
\begin{equation}\label{irregularidad}
s^*=\min\{j\in \I_{d_1} : Q_j\geq P_{j\coma k}  \peso{for every}
j\le k \le d_1 
\}
\end{equation}
We denote  $c=Q_{s^*}\,$. Therefore, by Eq. \eqref{promedios iniciales}, 
$$
c\, \sum_{i=s^*}^k\, h_i\geq \sum_{i=s^*}^k \, \alpha_i \peso{for every} s^*\leq k\leq d_1 \ .  
$$
In other words, $(\alpha_k)_{k=s^*}^n \prec (h_k\, c)_{k=s^*}^{d_1}\,$. 
If $s^*=1$ then we set $p=1$, $g_0=0$, $g_1=d_1$ and $\gamma_1=c>0$. Then 
items 1. and 2. of the statement are satisfied in this case.

\pausa
Otherwise, $s^*>1$ and we proceed to find the step $g_1$.
First, we define $\gamma_1$:
\[ \gamma_1=\max\{ P_{1,k}\, :\, 1\leq k\leq s^*-1\}\, , \]
and then we define $g_1$ as:
\[g_1=\max\{j\in \I_{s^*-1} \,: P_{1,j}=\gamma_1\}\,.\]
By construction, we obtain that
\[(\gamma_1h_i)_{i=1}^{g_1}\succ (\alpha_i)_{i=1}^{g_1}\,.\]
Now, if $g_1=s^*-1$ then we set $p=2$, $g_0=0$, $g_2=d_1$ and $\gamma_2=c>0$.

\pausa Otherwise, $g_1<s^*-1$ (and having the index $s^*$ fixed), we define $g_2$ in a similar way:
\[ \gamma_2=\max\{ P_{g_1+1,k}\, :\, g_1<k\leq s^*-1\}\]
and then, 
\[g_2=\max\{g_1+1\leq j\leq s^*-1 \,: P_{g_1+1,j}=\gamma_2\}\,.\]
Again, by construction we have that
\[(\gamma_2h_i)_{i=g_1+1}^{g_2}\succ (\alpha_i)_{i=g_1+1}^{g_2}\,.\]
We claim that $\gamma_1>\gamma_2$. Indeed, suppose that  $P_{g_1+1,g_2}=\gamma_2\geq \gamma_1=P_{1,g_1}$. Then,
\begin{align*}
P_{1,g_2}-P_{1,g_1}&=\frac{\sum_{i=1}^{g_1} \alpha_i + \sum_{i=g_1+1}^{g_2} \alpha_i}{ \sum_{i=1}^{g_1} h_i + \sum_{i=g_1+1}^{g_2} h_i}-\frac{\sum_{i=1}^{g_1} \alpha_i}{\sum_{i=1}^{g_1} h_i}=\\
&=\frac{\left(\sum_{i=1}^{g_1} h_i \right)\left(\sum_{i=g_1+1}^{g_2} \alpha_i\right)- \left(\sum_{i=g_1+1}^{g_2} h_i\right)\left(\sum_{i=1}^{g_1} \alpha_i\right)}{ \sum_{i=1}^{g_1} h_i\left(\sum_{i=1}^{g_1} h_i + \sum_{i=g_1+1}^{g_2} h_i\right)}\\
&=\frac{\left(\sum_{i=1}^{g_1} h_i\right)\left(\sum_{i=g_1+1}^{g_2} h_i\right)(\gamma_2-\gamma_1)}{ \sum_{i=1}^{g_1} h_i\left(\sum_{i=1}^{g_1} h_i + \sum_{i=g_1+1}^{g_2} h_i\right)}\geq 0\,.
\end{align*}
Hence, $P_{1,g_2}=P_{1,g_1}=\gamma_1$ which contradicts the definition of $g_1$, so the claim is proved.

\pausa
We can continue inductively with this process, that is, once we find $g_{k-1}<s^*-1$ 
we compute first $\gamma_k$ as the maximum among $P_{g_{k-1}+1, l}$, with $g_{k-1}+1\leq l\leq s^*-1$ and then define $g_k\leq s^*-1$ 
as the maximum index $g_{k-1}+1\leq l\leq s^*-1$ such that $P_{g_{k-1}+1, l}=\gamma_k$. As before, 
this construction guarantees the corresponding block majorization. 

\pausa
Notice that in the last step, corresponding to the $p-1$ iteration of the process, we  necessarily have $g_{p-1}=s^*-1$.
Define $\gamma_p=Q_{g_{p-1}+1}=c>0$ and $g_p=d_1$.

\pausa
By construction, and the previous remarks we have that 
$\gamma_1>\gamma_2>\ldots>\gamma_{p-1}$ and item \ref{mayo en los bloques} is satisfied. 
It remains to prove that $\gamma_{p-1}>\gamma_p\,$.

\pausa
Suppose, on the contrary, that $\gamma_p\geq \gamma_{p-1}$. Consider $\overline{c}=Q_{g_{p-2}+1}$. Clearly,  $\overline{c}$ is a convex combination of $\gamma_p$ and $\gamma_{p-1}$, say $\overline{c}=t\gamma_{p-1}+(1-t)\gamma_p\,$. In particular, $\gamma_{p-1}\leq \overline{c}\leq \gamma_p\,$. 
Therefore, we have that 
\begin{equation}\label{los promedios mas chicos que gp} 
P_{g_{p-2}+1,l}\leq\gamma_{p-1} \leq \overline{c}\coma \text{ for } g_{p-2}+1\leq l\leq g_{p-1}
\end{equation} 
\pausa
Let $g_{p-1}+1\leq l<d_1$, and denote by
$A=\sum_{i=g_{p-2}+1}^{g_{p-1}} h_i$, $B=\sum_{i=g_{p-1}+1}^l h_i$ and $C=\sum_{i=l+1}^{d_1} h_i$.
Notice that, with this notation, in the convex combination that generates $\overline{c}$, we have  $t=\frac{A}{A+B+C}\,$.
Then, since $\gamma_{p-1}\leq \gamma_{p}\,$,
\begin{align*}
\frac{A}{A+B}\, \gamma_{p-1}+\frac{B}{A+B}\,\gamma_p&\leq \frac{A}{A+B+C}\,\gamma_{p-1}+ \frac{B+C}{A+B+C}\,\gamma_p
\end{align*}
Hence, since by definition of $\gamma_p$ we have $P_{g_{p-1}+1\coma l}\leq \gamma_p\,$, we obtain
\begin{equation}
\frac{A}{A+B}\,\gamma_{p-1}+\frac{B}{A+B}P_{g_{p-1}+1\coma l}\leq \frac{A}{A+B+C}\,\gamma_{p-1}+ \frac{B+C}{A+B+C}\,\gamma_p
\end{equation}
since $\frac{A}{A+B}\,\gamma_{p-1}+\frac{B}{A+B}\, P_{g_{p-1}+1\coma l}=P_{g_{p-2}+1 \coma l}$ 
and $\frac{A}{A+B+C}\,\gamma_{p-1}+ \frac{B+C}{A+B+C}\,\gamma_p=\overline{c}$, we deduce
\begin{equation}\label{los promedios que ganan gp}
P_{g_{p-2}+1 \coma l}\leq \overline{c} \ \coma \peso{for} g_{p-1}+1\leq l\leq d_1
\end{equation}
Therefore, Eqs. \eqref{los promedios mas chicos que gp}, \eqref{los promedios que ganan gp} imply that, for $j=g_{p-2}+1<s^*=g_{p-1}+1$,
$$
P_{j\coma l}\leq Q_j \ \coma \peso{for} l=j,\ldots, d_1$$
which contradicts the construction of $s^*$.
So we can conclude that $\gamma_{p-1}>\gamma_p$ and the theorem is proved.
\end{proof}

\pausa Next, we introduce the following vectors associated to a pair $(\alpha\coma \dd)$.

\begin{fed}\label{fed los gamma ops}\rm
Let $\dd=(d_j)_{j\in\I_m}\in (\N^m)\da$ and $\alpha=(\alpha_i)_{i\in\I_n}\in (\R_{\geq 0}^n)\da$ be such that $d_1\leq n$.
Let $0=g_0<\cdots< g_p=d_1$ and $\gamma_1>\cdots>\gamma_p>0$ be as in Theorem \ref{defi del pegoteo de espectros}. Then, we set:
\ben
\item $\mu_1^{\rm op}= (\mu_{i\,1} ^{\rm op})_{i\in\I_{d_1}}=
(\gamma_k\uno_{g_k-g_{k-1}})_{k=1}^p\in (\R_{> 0}^{d_1})\da$;
\item For $j>1$  set $\mu_j^{\rm op}=  (\mu_{i\,j} ^{\rm op})_{i\in\I_{d_j}}\in (\R_{> 0}^{d_j})\da$ 
such that $\mu_{i\,j}^{\rm op}=\mu_{i\,1}^{\rm op}$ for $i\in \I_{d_j}\,$. 
\item Denote $\cM^{\rm op}=\{\mu_j^{\rm op}\}_{j\in \I_m}\in \prod_{j \in \I_m} \, (\R_{>0}^{d_j})\da$. 
\een
At this point, the vector $\mu_1^{\rm op}$ (and so  the sequence $\cM^{\rm op}$) depends on the choice of the indexes 
$0=g_0<\cdots< g_p=d_1$ from Theorem \ref{defi del pegoteo de espectros}. Nevertheless, we shall see 
now  that, when rearranged,  
$\cM^{\rm op}$ has minimality properties for majorization, so that they are univocally determined. 
\EOE
\end{fed}

\begin{rem}\label{rem del rem}
Let $\gamma_1\geq \ldots\geq \gamma_p\in \R$ and consider 
$\la=(\gamma_1\, \uno_{r_1},\ldots,\gamma_p\, \uno_{r_p})=(\la_i)_{i\in\I_r}\in (\R^r)\da$, 
where $r \igdef\sum_{i\in\I_p}r_i\,$. Set 
$s_k=\sum_{j\in\I_k} r_j\,$, for $k\in\I_{p}\,$. Given $\beta\in(\R^r)\da$ such that $\tr(\la)=\tr(\beta)$ 
then 
\beq\label{la 2.2}
\la\prec\beta \iff
\sum_{i\in\I_k} \gamma_i\, r_i\leq \sum_{j\in\I_{s_k}} \beta_j \ , \peso{for} k\in\I_{p-1}\ .
\eeq
Indeed, if the right conditions hold and there exists $0\leq k\leq p-1$ with $s_k < t <s_{k+1}$ 
($s_0 = 0$) and
such that $\sum\limits_{j \in \I_t} \la_j > \sum\limits_{j \in \I_t} \beta_j\,$, it is easy to see that 
$$
\sum_{j =s_k+1}^t \beta_j < \sum_{j =s_k+1}^{t} \la_j = (t-s_k) \, \gamma_{k+1}  \implies \beta_t < \gamma_{k+1}
\implies 
 \sum_{j\in\I_{s_{k+1}}} \beta_j  <\sum_{i\in\I_{k+1}} \gamma_i\, r_i \ ,
$$
which contradicts our assumption \eqref{la 2.2}. Therefore $\la \prec \beta$. 
\EOE
\end{rem}

\pausa We can now state our second main result.

\begin{teo}\label{teo el algo sirve}
Let $\dd=(d_j)_{j\in\I_m}\in (\N^m)\da$ and $\alpha=(\alpha_i)_{i\in\I_n}\in (\R_{\geq 0}^n)\da$ be such that $d_1\leq n$.
\ben
\item 
If we let $\cM^{\rm op}$ be as in Definition \ref{fed los gamma ops} then, 
$(\alpha\coma  \cM^{\rm op})$ is admissible. In particular, there 
exists $\Phi^{\rm op}=(\cF^{\rm op}_j)_{j\in\I_m}\in \mathcal D(\alpha\coma \dd)$ such that 
$\cM^{\rm op} = \cM_{\Phi^{\rm op}}\,$. 
\item  
If $\Phi=(\cF_j)_{j\in\I_m}\in \cD(\alpha\coma \dd)$, then 
$$ \Lambda_{\Phi^{\rm{op}}}\prec \Lambda_\Phi\, ,$$
where $\Lambda_{\Phi^{\rm{op}}}\coma \Lambda_\Phi \in \R_{\ge 0}^{d}$ are as in Definition \ref{rem equiv mayo}.
\een
\end{teo}
\begin{proof}
1. By Definition \ref{fed los gamma ops}, each vector $\mu_i^{\rm op}\oplus 0_{d_1-d_i} = 
\mu_1^{\rm op} \circ \big(\, \uno_{d_i}\oplus 0_{d_1-d_i} \, \big) $. Then 
$$
\sigma_{\cM^{\rm op}}=\sum_{i\in\I_m} (\mu_i^{\rm op}\oplus 0_{d_1-d_i})
= \mu_1^{\rm op} \circ \sum_{i=1}^m \  \big(\, \uno_{d_i}\oplus 0_{d_1-d_i} \, \big) 
\stackrel{\eqref{eq defi h}}{=} \mu_1^{\rm op}\circ h \ ,
$$
using the vector $h=(h_i)_{i\in \I_{d_1}}$ defined from $\dd$ as in Eqs. \eqref{eq defi hi} and \eqref{eq defi h}.
Therefore, by Eq. \eqref {admite} in Theorem \ref{defi del pegoteo de espectros}, $\sigma_{\cM^{\rm op}}\succ \alpha$, so the statement follows from Theorem \ref{teo: caract de la part mayo}.

\pausa
2. Let $\Phi=(\cF_j)_{j\in\I_m}\in\cD(\alpha\coma \dd)$ be such that $\cF_j=\{f_{ij}\}_{i\in\I_n}$, for $j\in\I_m$. 
Then 
\beq\label{ec teo1 1}
\sum_{j\in\I_m}\|f_{ij}\|^2=\alpha_i \peso{for} i\in\I_n\,.
\eeq
On the other hand, if we denote  $\la_{ij} = \la_i(S_{\cF_j})$ for $j\in\I_m$ and $i \in \I_{d_j} \,$,  we also have that 
$$ 
(\|f_{ij}\|^2)_{i\in\I_n}\prec \la(S_{\cF_j})= (\la_{ij})_{i\in\I_{d_j}}\peso{for} j\in\I_m\ .
$$
Hence, we conclude that
\beq\label{ec teo1 2}
\sum_{i\in\I_s} \|f_{ij}\|^2 \leq \sum_{i=1}^{\min\{ s\coma d_j\}} \la_{ij} \peso{for} s\in\I_n\py j\in\I_m\,.
\eeq
Let $\Phi^{\rm op}\in\cD(\alpha\coma \dd)$ be as in item 1. We also consider $p\in\N$,
$g_0=0<g_1<\ldots<g_p=d_1$ and $\gamma_1> \ldots> \gamma_p> 0$ 
 as in Theorem \ref{defi del pegoteo de espectros}. 
Let  $r_1,\ldots,r_p\in\N$ such that $\Lambda_{\Phi^{\rm op}}\da=(\gamma_\ell\, \uno_{r_\ell})_{\ell\in\I_p}$ 
as in Eq. \eqref{eq defi r}.
%
By Remark \ref{rem del rem}, 
in order to prove that 
$\Lambda_{\Phi^{\rm op}}
\prec\Lambda_\Phi$ it is sufficient to check that
\beq\label{ec teo1 3}
\sum_{\ell\in\I_q} r_\ell\, \gamma_\ell \leq 
\sum_{i\in \I_{s_q}} (\Lambda_{\Phi})_i\da 
 \ \ \ ,\peso{where} s_q=\sum_{\ell\in\I_q}r_\ell\ \ \ ,\peso{for every}
q\in\I_{p-1}\ ,
\eeq
because $\tr \,\Lambda_{\Phi^{\rm op}} = \tr \,\alpha = \tr\, \Lambda_\Phi\,$. 
Fix $q\in\I_{p-1}$ and consider the set
$$S_q=\{(i\coma j): \ 1\leq i\leq \min\{g_q\coma d_j\}\ , \ j\in\I_m\}\ . $$  
It is easy to see, using Eq. \eqref{eq defi h} 
(or looking at the rows and columns of the Figure 1), 
that 
\beq\label{cuantos}
\#S_q = \sum _{j\in\I_m} \min\{g_q\coma d_j\} 
\stackrel{\eqref{eq defi h}}{=} \sum_{i\in \I_{g_q}} h_i = 
%
\sum_{\ell\in\I_q} \left( \sum_{i=g_{\ell-1}+1}^{g_\ell}   h_i\right)  
\stackrel{\eqref{eq defi r}}{=} \sum_{\ell\in\I_q} r_\ell = s_q \ .
\eeq
Therefore we can show Eq. \eqref{ec teo1 3} as follows:   For every $q\in\I_{p-1}\,$, 
\begin{eqnarray*}
 \sum_{(i\coma j)\in S_q} \la_{ij}& =&\sum_{j\in \I_m} \sum_{i=1}^{\min\{g_q\coma d_j\}} \la_{ij}
\stackrel{\eqref{ec teo1 2}}{\ge} \sum_{j\in\I_m} \sum_{i\in\I_{g_q}} \|f_{ij}\|^2 
\stackrel{\eqref{ec teo1 1}}{=} \sum_{i\in\I_{g_q}} \alpha_i\\
&=&\sum_{\ell\in\I_q}\left(\sum_{i=g_{\ell-1}+1}^{g_\ell}\alpha_i\right)
\stackrel{\eqref{mayo en los bloques}}{=} 
\sum_{\ell\in\I_q}  \left( \sum_{i=g_{\ell-1}+1}^{g_\ell}  \gamma_\ell  \, h_i\right)
\stackrel{\eqref{eq defi r}}{=} \sum_{\ell\in\I_q}r_\ell\,\gamma_\ell \ . 
\end{eqnarray*}
Since $\suml_{i\in \I_{s_q}} (\Lambda_{\Phi})_i\da \stackrel{\eqref{cuantos}}{\ge}
\suml_{(i\coma j)\in S_q} \la_{ij} \,$, then 
Eq. \eqref{ec teo1 3} follows, and 
$\Lambda_{\Phi^{\rm op}}
\prec\Lambda_\Phi\,$.
\end{proof}

\pausa
Theorem \ref{teo el algo sirve} together with the argument in Remark \ref{rem equiv mayo25} allow us to obtain our 
third main result.

\begin{teo}\label{teo el algo sirve con potenciales}
Let $\dd=(d_j)_{j\in\I_m}\in (\N^m)\da$ and $\alpha=(\alpha_i)_{i\in\I_n}\in (\R_{\geq 0}^n)\da$ be such that $d_1\leq n$.
\ben
\item 
Let $\Phi^{\rm op}=(\cF_j^{\rm op})_{j\in\I_m}\in\cD(\alpha\coma \dd)$ be as in Theorem \ref{teo el algo sirve}. 
 If $\varphi\in\convf$ then we have that
\beq \label{ecuac cuac cuac}
\text{P}_\varphi (\Phi^{\rm op}) 
\leq
\pot(\Phi) \peso{for every} \Phi =(\cF_j)_{j\in\I_m}\in \cD(\alpha\coma \dd)\,.
\eeq
\item 
	Moreover, if $\Phi=(\cF_j)_{j\in\I_m}\in \cD(\alpha\coma \dd)$ is
such that there exists $\varphi\in\convfs$ for which
$\Phi $ is a global minimum for $\pot$ (i.e., if
equality holds in Eq. \eqref{ecuac cuac cuac}\,), then 
\beq\label{M unico}
\Lambda_{\Phi}=\Lambda_{\Phi^{\rm op}} \py 
\cM_\Phi = \cM^{\rm op} \ , \peso{so that}  
\la(S_{\cF_j})=\mu_j^{\rm op} \in (\R_{>0}^{d_j})\da \ , 
\eeq
and, in particular,  $\cF_j$ is a frame for $\C^{d_j}$ 
for every  $j\in\I_m  \,$. 
\een
\end{teo}
\proof
Let $\Phi=(\cF_j)_{j\in\I_m}\in \cD(\alpha\coma \dd)$, then by Theorem \ref{teo el algo sirve} 
we know that $\Lambda_{\Phi^{\rm op}}\prec \Lambda_{\Phi}\,$. Therefore, by 
Remark \ref{rem equiv mayo25} we get that for every $\varphi\in\convf$, 
$$
\pot (\Phi^{\rm op})=\sum_{j\in\I_m} \pot(\cF_j^{\rm op})=
\tr(\varphi(\Lambda_{\Phi^{\rm op}}))\leq \tr(\varphi(\Lambda_{\Phi}))=\sum_{j\in\I_m} \pot(\cF_j)=\pot (\Phi)\,.
$$
If $\varphi\in\convfs$ and $\Phi=(\cF_j)_{j\in\I_m}\in \cD(\alpha\coma \dd)$ is such that
equality holds in Eq. \eqref{ecuac cuac cuac}, let 
$$ 
\cC=\{ \Lambda_\Psi:\ \Psi=(\cG_j)_{j\in\I_m}\in \cD(\alpha\coma \dd)\}\ \inc \  \R_{\geq 0}^{d} \ .$$
By Lemma \ref{rem: primera equiv admis y part mayo} it follows that $\cC$ is a convex set. 
\pausa We finally introduce $F_\varphi:\cC\rightarrow \R_{\geq 0}$, $F_\varphi(\Lambda)
=\tr(\varphi(\Lambda))$ for $\Lambda\in\cC$.
Since $\varphi$ is strictly convex we immediately see that $F$ - which is defined 
on the convex set $\cC$ - is strictly convex as well.
Hence, there exists a {\it unique} $\Lambda_\varphi\in\cC$ such that 
$$ 
F(\Lambda_\varphi)=\min\{F(\Lambda):\ \Lambda\in\cC\}\  .
$$
Notice that by hypothesis, we have that $F(\Lambda_\Phi)=F(\Lambda_{\Phi^{\rm op}})=\min\{F(\Lambda):\ \Lambda\in\cC\}$ so then
\beq
(\lambda(S_{\cF_j}))_{j\in\I_m}=\Lambda_{\Phi}=\Lambda_\varphi=\Lambda_{\Phi^{\rm op}}=(\lambda(S_{\cF_j^{\rm op}}))_{j\in\I_m}\ .
\QEDP
\eeq

\begin{rem}\label{era unico}
 As a consequence of Eq.'s \eqref{ecuac cuac cuac} and \eqref{M unico}, the sequence $\cM^{\rm op}$ of Definition \ref{fed los gamma ops} 
and the indexes $(g_i)_{i\in \I_p}$ of Theorem \ref{defi del pegoteo de espectros} are univocally determined. 

\pausa
On the other hand, with the notation of Theorem \ref{teo el algo sirve con potenciales}, if we 
assume that $m=1$, then the previous 
theorem recovers the main results from \cite{CKFT,mr2010,MRS13,mrs2}. In this case, the optimal spectra $\mu^{\rm op}$ is obtained in terms 
of the water-filling construction. Hence, our results can be considered as a multivariated extension of the water-filling construction 
(see \cite{mrs2}).
\EOE
\end{rem}

\section{Final comments and examples}\label{sec ejems}

\subsection{On the weight partitions}

\pausa
By Theorem \ref{teo el algo sirve con potenciales} and Remark \ref{era unico}, 
the spectral structure of all $(\alpha\coma \dd)$-designs that minimize a
strictly convex potential 
on $\cD(\alpha\coma \dd)$ 
is unique. 
It is natural to wonder whether the $(\alpha\coma m)$-weight partitions corresponding to such 
minimizers also coincide. It turns out that this is not the case, as we shall see in the following example: 

\pausa
Let $\alpha=\uno_6\in (\R_{>0}^6)\da$, $m=2$ and let $\dd=(4,2)\in\N^2$. In this case $\cM^{\rm op}$ is given by 
$$ \mu_1^{\rm op}=\uno_4 \py \mu_2^{\rm op}=\uno_2\ ,
$$
since $\alpha =\uno_6\prec  (2\coma 2\coma 1\coma 1)=\sigma_{\cM^{\rm op}}$ (so that $\cM^{\rm op}$ is 
$(\alpha\coma \dd)$-admissible), and because
the associated vector  $\Lambda ^{\rm op}= (\mu_1^{\rm op}\coma \mu_2^{\rm op}) = \uno_6\,$ is 
minimal for majorization. Let 
%
$$
A_1 = \left(\,\frac{4}{6}\,\uno_6 \coma \frac{2}{6}\,\uno_6\,\right) \in P_{\alpha\coma 2} 
\inc \matrec{6\coma 2}  \py
A_2 = \big(\,c_1(A_2)\coma c_2(A_2)\,\big) \in P_{\alpha\coma 2} 
$$
where $c_1(A_2)=(\uno_4\coma 0_2)$ and $c_2(A_2)=(0_4\coma \uno_2)\,$.
It is easy to see that both matrices satisfy 
Eq. \eqref{col}  in Lemma \ref{rem: primera equiv admis y part mayo}.  
%
%
%
Thus, we can construct $\Phi^{1}=(\cF_1^{1}\coma\cF_2^{1})\in\cD(\alpha\coma \dd)$ 
with weight partition $A_1\,$, in such a way that 
$\cF_1^{1}$ is a Parseval frame for $\C^4$ and  $\cF_2^{1}$ is a Parseval frame for $\C^2$ (both of 6 vectors).
%
On the other hand, if we let $\{e^{(k)}_\ell\}_{\ell\in\I_k}$ 
denote the canonical basis of $\C^k$ for $k\in\N$ and let 
$\Phi^{2}=(\cF_1^{2}\coma\cF_2^{2})\in\cD(\alpha\coma \dd)$ with 
$$
\cF_1^{2}=\{e^{(4)}_1,\ldots,e^{(4)}_4,0,0\} \in (\C^4)^6
\py \cF_2^2=\{0,0,0,0,e^{(2)}_1,e^{(2)}_2\}\in (\C^2)^6 \ , 
$$
then $\Phi^{2}$ has weight partition $A_2\,$. 
Clearly $\cM_{\Phi_1} = \cM_{\Phi_2} = \cM^{\rm op}$. 
That is, $A_1$ and $A_2$ 
are both associated to optimal $(\alpha\coma \dd)$-designs. Thus, 
weight partitions inducing optimal $(\alpha\coma \dd)$-designs are not unique.
Note that the $(\alpha\coma \dd )$-designs $\Phi^{1}$ and $\Phi^2$ are qualitatively 
different. 
%
%
%
%
%

\begin{rem}\label{esteemm}
The fact that there are many $(\alpha\coma m)$-partition matrices $A \in P_{\alpha\coma m}$ that 
are associated to minimizers, as in the previous example, shows that the construction of  
optimal $(\alpha\coma \dd)$-designs can not be reduced to a convex optimization problem in 
the space $P_{\alpha\coma m}$ of $(\alpha\coma m)$-weight partitions.

\pausa
We remark that in a previous version of this paper we constructed an algorithm which produced a particular 
matrix $A^{\rm op} \in P_{\alpha\coma m}$ associated to a minimizer. That is, once $A^{\rm op}$ was constructed
we considered the so-called water-filling of the columns $c_j(A^{\rm op})$ 
in dimension $d_j$ (see \cite{mrs3}), which lead to the optimal spectra $\mu_j^{\rm op}$, 
for $j\in\I_m\,$. 

\pausa 
Our new strategy, based on Theorem \ref{teo: caract de la part mayo}, allow us to compute directly  the  optimal spectra
$\mu_j^{\rm op}$ for $j\in\I_m$ (and to show 
the existence of $(\alpha\coma \dd)$-designs with these spectra). Once the optimal spectra are 
computed then, using Remark \ref{rem: algotutti},
we can compute several associated weight partitions that in turn allow us to compute optimal $(\alpha\coma\dd)$-designs in an effective way (see Section \ref{sec exas} below).
This new approach has decreased considerably the length of 
the exposition of our results herein.\EOE
\end{rem}

%
%

\subsection{A compact description of the problem}\label{sec 42}
There is a reformulation of the problems of this paper in a more concise model. Let $\al$ 
and $\dd$ be as in Definition 
\ref{nueva def}. 
Set $d = \tr \, \dd$ and assume that $\cH= \C^{d} = \bigoplus_{j \in \I_m} \cH_j $ for some 
subspaces with $\dim \, \cH_j = d_j $, for $j \in \I_m\,$. Let us denote by $P_j : \cH \to \cH_j\inc \cH$ 
the corresponding projections. 

\pausa 
Notice that a sequence $\cG = \{g_i\}_{i\in\I_n} \in \cB_{\al\coma d}\inc \cH^n \iff$ the sequence $\Phi= (\cF_j)_{j\in\I_m} $ determined by 
$\cF_j = P_j (\cG) $ (i.e. $f_{ij} = P_j (g_i)\in \cH_j \cong \C^{d_j}$, $i \in \I_n$) for $j\in\I_m$, 
satisfies that $\Phi \in \mathcal D(\alpha\coma \dd)$. 

\pausa 
Consider the pinching map $\cC_\dd : \matrec{d}\to \matrec{d}$ given by 
$\cC_\dd(A) = \sum_{j \in \I_m} P_j \, A \, P_j\,$, for every $A \in 
\matrec{d}$. Then, for each $\varp \in \convf$  we can define a $\dd$-pinched potential 
$$
\potd (\cG) \igdef \tr \varp (\cC_\dd(S_\cG)\,) \peso{for every} \cG\in \cH^n \ ,
$$
which describes simultaneously the behavior of the projections of $\cG$ to each subspace $\cH_j\,$. 
Actually, with the previous notations, 
$$
\potd (\cG) = \sum_{j\in\I_m} \tr \varp (P_j \, S_\cG\, P_j) = 
\sum_{j\in\I_m} \pot(\cF_j)=\pot(\Phi) \ .
$$
Therefore the problem of finding optimal $(\alpha\coma \dd)$-designs (and studying their properties) translates to 
the study of sequences $\cG \in \cB_{\al\coma d}$ which minimize the $\dd$-pinched potentials $\potd\,$. 

\pausa We point out that for $\varphi\in\convf$ and $\cG\in\hil^n$
$$
\potd(\cG) \le \tr \, \varphi(S_\cG) \peso{but}
\potd(\cG) \neq \tr \, \varphi(S_\cG)  \peso{in general} 
$$ 
(see Definition \ref{pot generales}).
Therefore, previous results related with the structure of minimizers of convex potentials in 
$\cB_{\alpha\coma d}$ (e.g. \cite{mr2010}) do not apply 
to the $\dd$-pinched potential and we require a new approach to study this problem, as shown in Example \ref{exa: el primero}. In the paper we use the 
more complicated notation of sequences because it has proved to be more useful for all the 
computations detailed before. 

\subsection{Examples}\label{sec exas}
\pausa
Theorems \ref{teo: caract de la part mayo}, \ref{teo el algo sirve} and \ref{teo el algo sirve con potenciales}, 
combined with Remark \ref{rem: algotutti} allow us to describe a finite step algorithmic process 
for the effective construction of optimal designs $\Phi^{\rm op}\in\cD(\alpha\coma \dd)$ from the initial data $(\alpha, \dd)$.

\pausa
A possible scheme for the algorithmic procedure would be as follows:

\begin{algo}\label{algo1}
IMPUT DATA: $\left(\alpha\coma \dd\right)$.
\begin{description}
\item[STEP 1.] Along with the computation of the vector $h$, by means of an iterative process, we compute the value of $s^*$ defined 
in Theorem \ref{defi del pegoteo de espectros}. If $s^*=1$, we set $p=1$, $g_1=1$ and $\gamma_1=Q_1$. In this case, $\gamma_1\uno_{d_1}\circ h\succ \alpha$ and go to STEP 3. Otherwise, set $j=1$ and continue with
\item[STEP 2.] By computing the maximum among the means $P_{j,k}$, $j\leq k\leq s^*-1$ we compute $g_j$ and $\gamma_j$ as it is indicated in 
Theorem \ref{defi del pegoteo de espectros}. If $g_j=s^*-1$, rename $j=p-1$, and $\gamma_p=Q_{g_{p-1}+1}$ and go to STEP 3.
Otherwise, set $j=g_j+1$ and repeat STEP 2. This step produces the set of optimal spectra 
$\cM^{\rm op}=\{\mu_j^{\rm op}\}_{j\in \I_m}$, such that $(\alpha\coma \cM)$ is admissible.
\item[STEP 3.] Apply the finite step algorithm described in Remark \ref{rem: algotutti} and obtain $\Phi^{\rm op}=(\cF^{\rm op}_j)_{j\in \I_m}$ 
(as described in Theorem \ref{teo el algo sirve}) as output.
\end{description}
\end{algo}

%
%

\pausa
The following examples were obtained via an implementation of Algorithm \ref{algo1} using MATLAB.
\begin{exa}\label{exa: el primero}
Consider the family of weights given by $\alpha=\{9, 8, 7, 5, 4, 2.5, 2, 2, 1.5, 0.6, 0.5\}$ and suppose that the dimensions to be considered are $\dd= (7\coma 5\coma 3)$. 
In this case, the optimal spectra $\cM^{\rm op}$ are determined, as in Definition \ref{fed los gamma ops},  by 
$$ \mu^{\rm op}_1=( 3,   2.7583,    2.7583,    2.7583,    2.7583,    2.7583,    2.7583) \ .
$$ 
%
If $\sigma_{\cM^{\rm op}}\in\R^7$ is defined as in Theorem \ref{teo: caract de la part mayo},  
then, $\alpha\prec \sigma_{\cM^{\rm op}}\,$ by Theorem \ref{teo el algo sirve}. 
Using Remark \ref{rem: algotutti} we construct $D\in\mathcal {DS}(11)$ such that 
$D(\sigma_{\cM^{\rm op}} \oplus 0_4) =\alpha$. Setting $A\in \cM_{11,3}(\R_{\geq 0})$, 
such that $c_j(A)=D\, (\mu_j^{\rm op}\oplus 0_{11-d_j})\in\R_{\geq 0}^{11}$, for $j=1,2,3$,	
we get (for example) the following partition of $\alpha$: 
\[A=\footnotesize{\left[\begin{array}{rrr}
    3         &3         &3\\
    2.7583    &2.7583    &2.4833\\
    2.7583    &2.7583    &1.4833\\
    2.7583    &1.8135    &0.4282\\
    2.5267    &1.1307    &0.3425\\
    1.5792    &0.7067    &0.2141\\
    1.2634    &0.5654    &0.1713\\
    1.2634    &0.5654    &0.1713\\
    0.9475    &0.4240    &0.1285\\
    0.3790    &0.1696    &0.0514\\
    0.3158    &0.1413    &0.0428
\end{array}
\right]}\in P_{\alpha\coma 3}\,.
\]
Once we have the partitions and optimal spectra, we can construct examples of frames using these data, applying known algorithms like one-sided Bendel-Mickey algorithm (see \cite{CFMP,CMTL,ID,FWW}):
\[\cF_1=\tiny{\left[\begin{array}{rrrrrrrrrrr}
      0.0705   & 0.1956   &-0.0616  & -0.6865   &-0.6865    &0.3994   &-0.0845   &-0.3230   &-1.1553  &0.2649   &-0.3180\\
    0.2804   &-0.2311   &-0.2142    &0.2434    &0.2434   &-0.4716    &1.2808    &0.2534   &-0.4197  &0.3309   &-0.5206\\
    0.0380   &-0.1106   &-0.5728   &-0.8134   &-0.8134   &-0.2257    &0.3005    &0.0342    &0.9482  &0.1009   &-0.2125\\
   -0.0004   &-0.1760   &-0.3643   &-0.2125   &-0.2125   &-0.3592    &0.2804    &0.2989   &-0.4753 &-1.0345    &0.9956\\
   -0.4655   &-0.4260   &-0.5815    &0.0796    &0.0796   &-0.8695   &-0.8294    &0.3134   &-0.3310  &0.0127   &-0.6235\\
    0.1120    &0.2501    &0.3128   &-0.1034   &-0.1034    &0.5106   &-0.0368    &1.2019    &0.0419 &-0.6107   &-0.7246\\
    0.0391   &-0.0112    &0.0316    &0.0949    &0.0949   &-0.0229    &0.1448   &-0.9781    &0.1061 &-1.0607   &-0.8232
\end{array}
\right]}
\]
\[\cF_2=\tiny{\left[\begin{array}{rrrrrrrrrrr}
 0.1841    &0.2017    &0.3189    &0.0682   &-0.2093   &-0.2340   &-0.2960   &0.6595    &0.5432    &0.0340    &1.3437\\
    0.0249    &0.0273    &0.0432    &0.6049    &0.6893    &0.7707    &0.9748    &0.0893    &0.4598    &0.3451    &0.1842\\
   -0.1947   &-0.2132   &-0.3372   &-0.3744   &-0.1430   &-0.1599   &-0.2022   &-0.6973    &1.2517    &0.5169   &-0.1238\\
   -0.2625   &-0.2876   &-0.4547   &-0.2253    &0.1440    &0.1610    &0.2037   &-0.9404   &-0.4997   &-0.3351    &1.0506\\
   -0.0015   &-0.0016   &-0.0025   &-0.0619   &-0.0723   &-0.0808   &-0.1022   &-0.0053   &-0.6598    &1.5029    &0.2041
\end{array}
\right]}
\]

\[\cF_3=\tiny{\left[\begin{array}{rrrrrrrrrrr}
-0.0342  & -0.0375  & -0.0593  & -0.3714   &-0.3952   &-0.4419   &-0.5590   &-0.6249   &-1.1632   &-0.2605   &-0.3888\\
   -0.1953   &-0.2139   &-0.3383   &-0.1805    &0.0479    &0.0536    &0.0678    &0.0758    &0.1410   &-1.4873    &0.5519\\
    0.0592    &0.0649    &0.1026   &-0.0281   &-0.1129   &-0.1263   &-0.1597   &-0.1786   &-0.3324    &0.4511    &1.5950
\end{array}
\right]}
\]
Let $\Phi^{\rm op} = (\cF_1\coma \cF_2\coma \cF_3)\in\cD(\alpha\coma \dd)$. Then, by Theorem \ref{teo el algo sirve con potenciales} we have that
$$
\min\{\pot(\Phi)\ : \ \  \Phi \in \cD(\alpha\coma \dd)\}=\pot(\Phi^{\rm op}) = 3\ \varphi(3) + 
12\ \varphi(2.7583)\,.
$$
Now, with the notation and terminology of Section \ref{sec 42}, we get the following lower bound for the $\dd$-pinched potential (notice that $d=d_1+d_2+d_3=15$) 
\beq \label{eq min pot con restric12} 
\min \{\tr\, \varp (S_ \cG)\ :\  \ \cG\in  \cB_{\al\coma 15}\} \ge 
\min \{\potd(\cG)\ :\  \ \cG\in  \cB_{\al\coma 15}\}= 3\ \varphi(3) + 
12\ \varphi(2.7583)\,.
\eeq
%
%
Indeed, since $d=15> n=11$, 
we have that (see \cite{mr2010})
\beq \label{eq min pot sin restric12}
\min \{\tr\, \varp (S_ \cG)\ :\  \ \cG\in  \cB_{\al\coma 15}\}=\sum_{j\in\I_{11}} \varphi(\alpha_j)+ 4\, \varphi(0)\,.
\eeq Moreover, in case $\varphi\in\convfs$ then the minimizers of the potential 
$\cG\mapsto \tr\, \varp (S_ \cG)$ in $\cB_{\al\coma 15}$ 
are sequences $\cG=\{ g_i\}_{i\in\I_{11}}\in\cH^{11}$ of mutually orthogonal vectors.
If we further choose $\varphi(x)=x^2$, $x\geq 0$, then the reader can check that nor the minimal value, 
nor the geometric structure of minimizers of Eqs. \eqref{eq min pot con restric12} and \eqref{eq min pot sin restric12} coincide.
\EOE
\end{exa} 


\begin{exa}
When $\alpha=\{20, 19.5, 10, 5, 4.5, 3, 2.4, 2\}$ and $\dd=\{ 5,     4,     4,     3,     2\}$, Algorithm  \ref{algo1} 
constructs the optimal spectra given by
\begin{align*}
 \mu^{\rm op}_1&=( 4,    3.9,   3.3625,    3.3625,    3.3625)\\
 \mu^{\rm op}_2&=(4,    3.9,   3.3625,    3.3625)\\
 \mu^{\rm op}_3&=(4,    3.9,   3.3625,    3.3625)\\
 \mu^{\rm op}_4&=(4,    3.9,   3.3625)\\
 \mu^{\rm op}_5&=(4,    3.9)\\
\end{align*}
where the smaller spectrum does not have the constant $3.3625$.
As before, using Remark \ref{rem: algotutti}, we obtain the following partition:
\[A=\footnotesize{\left[\begin{array}{rrrrr}
    4         &4         &4         &4         &4  \\
    3.9       &3.9       &3.9       &3.9       &3.9\\
    3.3625    &2.8875    &2.5       &1.25      &     0\\
    1.9896    &1.1354    &1.25      &0.625     &     0\\
    1.7907    &1.0218    &1.125     &0.5625    &     0\\
    1.1938    &0.6812    &0.75      &0.375     &     0\\
    0.955     &0.545     &0.6       &0.3       &     0\\
    0.7959    &0.4541    &0.5       &0.25      &     0
\end{array}
\right]}\in P_{\alpha \coma 5}\,.
\]\EOE
\end{exa}
\pausa
{\bf Acknowledgements:} We would like to thank the reviewers for several comments and suggestions that helped us to improve the contents of this manuscript.


{\scriptsize
\begin{thebibliography}{99}

\bibitem{Illi} J. Antezana, P. Massey, M. Ruiz and D. Stojanoff,
The Schur-Horn theorem for operators and frames with prescribed norms and frame operator,
 Illinois J.  Math., { 51} (2007),  537-560.


\bibitem{BMS1} M.J. Benac, P. Massey, D. Stojanoff, 
Convex potentials and optimal shift generated oblique duals in shift invariant spaces. 
J. Fourier Anal. Appl. 23 (2017), no. 2, 401-441. 

\bibitem{BMS2} M.J. Benac, P. Massey, D. Stojanoff, 
 Frames of translates with prescribed fine structure in shift invariant spaces. 
J. Funct. Anal. 271 (2016), no. 9, 2631-2671. 


\bibitem{BF} J.J. Benedetto, M. Fickus, Finite normalized tight frames. Frames. Adv. Comput. Math. 18 (2003), no. 2-4, 357-385.

\bibitem{Bhat} R. Bhatia,  Matrix Analysis,
Berlin-Heildelberg-New York, Springer 1997.


\bibitem{BodPau} B.G. Bodmann, V.I. Paulsen, Frames, graphs and erasures. Linear Algebra Appl. 404 (2005), 118-146.

\bibitem{TaF} P.G. Casazza, The art of frame theory, Taiwanese J. Math. 4 (2000), no. 2, 129-201.

\bibitem{Casagregado}P.G. Casazza, Custom building finite frames. In Wavelets, frames
and operator theory, volume 345 of Contemp. Math., Amer. Math. Soc., Providence, RI, 2004, 61-86.

\bibitem{CKFT} P.G. Casazza, M. Fickus, J. Kovacevic, M.T. Leon, J.C. Tremain, A physical interpretation of tight frames. Harmonic analysis and applications, 51--76, Appl. Numer. Harmon. Anal., Birkhäuser Boston, MA, 2006.

\bibitem{CFMP} J. Cahill, M. Fickus, D.G. Mixon, M.J. Poteet, N. Strawn, Constructing finite frames of a given spectrum and set of lengths, Appl. Comput. Harmon. Anal. 35 (2013), 52-73.


\bibitem{CMTL} P.G. Casazza, and M.T. Leon, Existence and construction of finite frames with a given frame operator. Int. J. Pure Appl. Math. 63 (2010), no. 2, 149-157.


\bibitem{FinFram} P. G. Casazza and G. Kutyniok eds., Finite Frames: Theory and Applications. Birkhauser, 2012. xii + 483 pp.

\bibitem{Chr} O. Christensen, An introduction to frames and Riesz bases. Applied and Numerical Harmonic Analysis. Birkhäuser Boston, 2003. xxii+440 pp.


\bibitem{ID} I.S. Dhillon, R.W. Heath Jr., M.A. Sustik, J.A. Tropp, Generalized finite algorithms for constructing Hermitian matrices with prescribed diagonal and spectrum, SIAM J. Matrix Anal. Appl. 27 (1) (2005) 61-71.

\bibitem{DFKLOW} K. Dykema, D. Freeman, K. Kornelson, D. Larson, M. Ordower, E. Weber, Ellipsoidal tight frames and projection decomposition of operators: Illinois J. Math. 48 (2004), 477-489.

\bibitem{FWW} D. J. Feng, L. Wang and Y. Wang, Generation of finite tight frames by Householder transformations. Adv Comput Math 24 (2006), 297-309.


\bibitem{FMaP}
M. Fickus, J. Marks, M. Poteet, A generalized Schur-Horn theorem
and optimal frame completions. Appl. Comput. Harmon. Anal. 40
(2016), no. 3, 505-528.

\bibitem{FMP}
M. Fickus, D. G. Mixon and M. J. Poteet,  Frame completions for
optimally robust reconstruction, Proceedings of SPIE, 8138:
81380Q/1-8 (2011).

\bibitem{HS07} B. Hassibi, M. Sharif, Fundamental Limits in MIMO Broadcast Channels. IEEE Journal on Selected Areas in Communications
25(7) (2007),  1333-1344.  

\bibitem{HolPau} R.B. Holmes, V.I. Paulsen, Optimal frames for erasures. Linear Algebra Appl. 377 (2004), 31-51.


\bibitem{KLagregado}K. A. Kornelson, D. R. Larson, Rank-one decomposition of operators and construction of frames. Wavelets, frames and operator theory, Contemp. Math.,
345, Amer. Math. Soc., Providence, RI, 2004, 203-214.

\bibitem{MRiS1}P. Massey, N. Rios , D. Stojanoff, Frame completions with prescribed norms: local minimizers and
applications.Adv. Comput. Math., in press.


\bibitem{MR0}P. Massey, M.A. Ruiz, Tight frame completions with prescribed norms. Sampl. Theory Signal Image Process. 7 (2008), no. 1, 1-13.


\bibitem{mr2010}
P. Massey, M. Ruiz; Minimization of convex functionals over frame operators. Adv. Comput. Math. 32 (2010), no. 2, 131-153.


\bibitem{MRS13} P. Massey, M. Ruiz , D. Stojanoff,  Optimal dual frames and frame completions for majorization. Appl. Comput. Harmon. Anal. 34 (2013), no. 2, 201-223.


\bibitem {mrs2}
P.G. Massey, M.A. Ruiz, D. Stojanoff; Optimal frame completions.
Advances in Computational Mathematics 40 (2014), 1011-1042.

\bibitem {mrs3}
P. Massey, M. Ruiz, D. Stojanoff; Optimal frame completions with prescribed norms for majorization.
J. Fourier Anal. Appl. 20 (2014), no. 5, 1111-1140.





\end{thebibliography}}
\end{document}